\documentclass{amsart}
\usepackage[utf8]{inputenc}

\usepackage[margin=1in]{geometry}
\usepackage[expansion=false]{microtype}

\usepackage{amsmath, amsfonts, amssymb, amsthm, amsopn, amscd}
\usepackage{eucal}

\usepackage{booktabs,enumitem}
\setlist[enumerate]{label=(\roman*),font=\normalfont}

\usepackage{tikz}
\usetikzlibrary{matrix}
\usepackage{tikz-cd}

\usepackage[pdfusetitle,unicode]{hyperref}

\numberwithin{equation}{section}

\theoremstyle{plain}
\newtheorem*{theorem*}{Theorem}
\newtheorem{theorem}[equation]{Theorem}

\newtheorem{proposition}[equation]{Proposition}
\newtheorem{lemma}[equation]{Lemma}
\newtheorem{corollary}[equation]{Corollary}

\theoremstyle{definition}
\newtheorem{definition}[equation]{Definition}
\newtheorem{construction}[equation]{Construction}
\newtheorem{example}[equation]{Example}

\newtheorem{remark}[equation]{Remark}

\let\scr=\mathcal
\let\bb=\mathbb
\let\rm=\mathrm

\def\Z{\bb Z}

\def\Q{\bb Q}

\def\A{\bb A}
\def\P{\bb P}
\def\V{\bb V}
\def\1{\mathbf 1}
\def\h{\mathrm h}

\def\G{\mathbb G}
\def\ph{\mathord-}
\makeatletter
\def\pt{{\mathpalette\pt@{.75}}}
\def\pt@#1#2{\mathord{\scalebox{#2}{$\m@th#1\bullet$}}}
\makeatother

\def\L{\mathrm{L}{}}
\def\LB{\mathrm{LB}{}}

\let\into=\hookrightarrow
\let\onto=\twoheadrightarrow
\def\simto{\xrightarrow{\sim}}

\def\suchthat{\:\vert\:}
\let\emptyset=\varnothing

\DeclareMathOperator{\Sym}{Sym}

\def\id{\mathrm{id}}

\def\Hom{\mathrm{Hom}}
\def\iHom{\underline{\Hom}}

\def\Aut{\mathrm{Aut}}
\def\Map{\mathrm{Map}}

\DeclareMathOperator{\Spec}{Spec}
\DeclareMathOperator{\Proj}{Proj}
\def\Th{\mathrm{Th}}
\def\Pic{\mathrm{Pic}}

\def\M{\mathrm{M}}

\def\B{\mathbb{B}}

\def\E{\mathbb{E}}

\DeclareMathOperator{\fib}{fib}
\DeclareMathOperator{\cofib}{cofib}

\DeclareMathOperator{\rk}{rk}

\def\Zar{\mathrm{Zar}}

\def\et{\mathrm{\acute et}}
\def\hyp{\mathrm{hyp}}

\def\Bl{\mathrm{Bl}}

\let\cat=\mathrm

\def\Gr{\mathrm{Gr}{}}

\def\BGL{\mathrm{BGL}}
\def\MGL{\mathrm{MGL}}
\def\PMGL{\mathrm{PMGL}}
\def\KGL{\mathrm{KGL}}
\def\K{\mathrm{K}{}}

\def\Mod{\cat{M}\mathrm{od}{}}
\def\Sch{\cat{S}\mathrm{ch}{}}

\def\Sm{{\cat{S}\mathrm{m}}}

\def\Fin{\cat F\mathrm{in}}
\def\op{\mathrm{op}}

\def\Vect{\cat{V}\mathrm{ect}{}}

\def\MotSp{\mathrm{MS}}
\def\Ani{\mathrm{Ani}}
\def\Sp{\mathrm{Sp}}

\def\Pr{\mathrm{P}\mathrm{r}}
\def\St{\mathrm{St}}

\def\CAlg{\mathrm{CAlg}{}}

\def\ev{\mathrm{ev}}
\def\coev{\mathrm{coev}}

\def\Cat{\mathrm{C}\mathrm{at}{}}

\def\Ab{\mathrm{A}\mathrm{b}}

\def\GL{\mathrm{GL}}

\def\Tw{\mathrm{Tw}{}}
\def\fp{\mathrm{fp}}
\def\Fun{\mathrm{Fun}}

\def\st{\mathrm{st}}

\def\Span{\mathrm{Span}}

\def\fin{\mathrm{fin}}

\def\Pro{\mathrm{Pro}}

\def\epi{\mathrm{epi}}
\def\ex{\mathrm{ex}}
\def\sbu{\mathrm{sbu}}
\def\ebu{\mathrm{ebu}}

\def\un{\mathrm{un}}
\def\fd{\mathrm{fd}}
\def\sncd{\mathrm{sncd}}
\def\Cube{\mathrm{Cube}}

\newcommand{\SpP}{\mathrm{Sp}_{\mathbb{P}^1}}
\newcommand{\SigmaP}{\Sigma_{\mathbb{P}^1}}
\newcommand{\Ori}{\mathrm{Ori}}
\newcommand{\TOri}{\mathrm{TOri}}
\newcommand{\Sel}{\mathrm{Sel}}

\newcommand{\PCob}{\underline{\Omega}}

\let\lim=\relax
\DeclareMathOperator*{\lim}{lim}

\DeclareMathOperator*{\colim}{colim}

\let\phi=\varphi
\let\epsilon=\varepsilon

\title{Algebraic cobordism and a Conner--Floyd isomorphism for algebraic K-theory}
\date{\today}

\author{Toni Annala}
\address{School of Mathematics\\
Institute for Advanced Study\\
1 Einstein Drive,
08540 Princeton, NJ, USA
}
\email{\href{mailto:tannala@ias.edu}{tannala@ias.edu}}
\urladdr{\url{https://www.math.ias.edu/~tannala/}}
\thanks{This article was written when T.A. was in residence at the Institute for Advanced Study in Princeton. }

\author{Marc Hoyois}
\address{Fakultät für Mathematik\\
Universität Regensburg\\
Universitätsstr. 31\\
93040 Regensburg\\
Germany}
\email{\href{mailto:marc.hoyois@ur.de}{marc.hoyois@ur.de}}
\urladdr{\url{https://hoyois.app.uni-regensburg.de}}
\thanks{M.H.\ was partially supported by the Collaborative Research Center SFB 1085 \emph{Higher Invariants} funded by the DFG}

\author{Ryomei Iwasa}
\address{Laboratoire de Math\'ematiques d'Orsay, Universit\'e Paris-Saclay, 307 rue Michel Magat, F-91405 Orsay.}
\email{\href{mailto:ryomei.iwasa@cnrs.fr}{ryomei.iwasa@cnrs.fr}}
\urladdr{\url{http://ryomei.com}}
\thanks{R.I.\ was supported by the European Research Council (ERC) under the European Union’s Horizon 2020 research and innovation programme (grant agreement No. 101001474).}

\begin{document}
	
\maketitle

\begin{abstract}
	We formulate and prove a Conner--Floyd isomorphism for the algebraic K-theory of arbitrary qcqs derived schemes.
	To that end, we study a stable $\infty$-category of non-$\A^1$-invariant motivic spectra, which turns out to be equivalent to the $\infty$-category of fundamental motivic spectra satisfying elementary blowup excision, previously introduced by the first and third authors.
	 We prove that this $\infty$-category satisfies $\P^1$-homotopy invariance and weighted $\A^1$-homotopy invariance, which we use in place of $\A^1$-homotopy invariance to obtain analogues of several key results from $\A^1$-homotopy theory.
	 These allow us in particular to define a universal oriented motivic $\E_\infty$-ring spectrum $\MGL$.
	We then prove that the algebraic K-theory of a qcqs derived scheme $X$ can be recovered from its $\MGL$-cohomology via a Conner--Floyd isomorphism
	\[
	\MGL^{**}(X)\otimes_{\L}\Z[\beta^{\pm 1}]\simeq \K^{**}(X),
	\]
	where $\L$ is the Lazard ring and $\K^{p,q}(X)=\K_{2q-p}(X)$. Finally, we prove a Snaith theorem for the periodized version of $\MGL$.
\end{abstract}

\tableofcontents

\section{Introduction}

In this article we construct \textit{algebraic cobordism} as a non-$\A^1$-invariant cohomology theory on derived schemes, and establish its basic expected properties: we show that it is the universal oriented cohomology theory, that it is related to algebraic K-theory via a Conner--Floyd isomorphism, and that its periodic version can be obtained from the infinite Grassmannian by inverting the Bott element (Snaith theorem).
These results refine the analogous theorems in $\A^1$-homotopy theory proven in \cite{Panin:2008}, \cite{Spitzweck-Bott}, and \cite{GepnerSnaith}, respectively. 
To establish our results, we study a stable $\infty$-category of non-$\A^1$-invariant motivic spectra as in \cite{AnnalaIwasa2}, which contains the stable $\A^1$-homotopy category of Morel--Voevodsky as its full subcategory of $\A^1$-invariant objects. We prove in particular that this category satisfies $\P^1$-homotopy invariance and weighted $\A^1$-homotopy invariance, which are weak forms of $\A^1$-homotopy invariance, allowing us to do ``homotopy theory'' in algebraic geometry while keeping the affine line $\A^1$ non-contractible. For example, we prove that the stack of vector bundles $\BGL_n$ is equivalent to the infinite Grassmannian $\Gr_n$ in this setting.

The $\A^1$-homotopy theory of Morel--Voevodsky \cite{MV} has been highly successful in the study of $\A^1$-invariant cohomology theories, playing an instrumental role in the resolution of the Milnor \cite{OVV, VV} and the Bloch--Kato \cite{Voevodsky:2008} conjectures. On the other hand, $\A^1$-homotopy theory is not useful for studying $p$-adic cohomology theories like crystalline and prismatic cohomology \cite{BMS,BS:2022}, since these are usually not $\A^1$-invariant. This is unfortunate, because deeper understanding of the $p$-torsion is often important for various applications: for instance, cohomological obstructions to the existence of resolution of singularities by blowups in characteristic $p$, if they exist, are expected to be $p$-torsion due to the existence of resolution by $p$-alterations \cite{sullivan:2004,temkin:2017}. Other important examples of non-$\A^1$-invariant cohomology theories are the algebraic and hermitian K-theory of singular schemes \cite{cortinas:2007,Schlichting}.

Here, we continue to develop the framework for non-$\A^1$-invariant motivic homotopy theory introduced in \cite{AnnalaIwasa2}, based on the idea of replacing $\A^1$-invariance with (a non-oriented version of) the projective bundle formula. More precisely, for a derived scheme $S$, we say that a Zariski sheaf on the $\infty$-category $\Sm_S$ of smooth $S$-schemes satisfies \textit{elementary blowup excision} if it carries the blowup square
\[
\begin{tikzcd}
	\P^{n-1}_X \ar[r] \ar[d] & \mathrm{Bl}_{\{0\}}(\A^n)_X \ar[d] \\
	\{0\}_X \ar[r] & \A^n_X
\end{tikzcd}
\]
to a cartesian square for every $X\in\Sm_S$ and $n\ge 1$.
Let $\scr{P}_{\Zar,\ebu}(\Sm_S,\Sp)$ denote the $\infty$-category of Zariski sheaves of spectra on $\Sm_S$ satisfying elementary blowup excision. Then, for the purposes of this paper, we define the \emph{$\infty$-category of motivic spectra} as the presentably symmetric monoidal $\infty$-category obtained from the latter by $\otimes$-inverting the pointed projective line $\P^1$:
\[
	\MotSp_S = \scr{P}_{\Zar,\ebu}(\Sm_S,\Sp)[(\P^1)^{-1}]\in\CAlg(\Pr^\L).
\]
The ultimate goal of our framework is to provide efficient tools to study non-$\A^1$-invariant cohomology theories and their cohomology operations. As most cohomology theories in algebraic geometry, including all the ones mentioned above, satisfy the projective bundle formula, this framework is widely applicable.

The main object of interest in this paper is algebraic cobordism. Our treatment of it follows closely the now classical treatment in $\A^1$-homotopy theory \cite{Voevodsky:1996,VoevodskyICM}. Namely, we define algebraic cobordism as the cohomology theory represented by a non-$\A^1$-invariant Thom spectrum $\MGL$. For a finite locally free sheaf $\scr{E}$ on $S$, we define the \textit{Thom space} of $\scr{E}$ by
\[
	\Th_S(\scr{E}) = \P(\scr{E}\oplus\scr{O})/\P(\scr{E}) \in \scr{P}(\Sm_S)_*.
\]
An important technical point is that we are able to promote $\Th_S(\ph)$ to a symmetric monoidal functor
\[
	\Th_S \colon \Vect(S) \to \scr{P}_{\Zar,\ebu}(\Sm_S)_*,
\]
which factors through the K-theory space $\K(S)$ after inverting $\P^1$ on the target.
Following \cite[Section 16]{norms}, we then define the \textit{algebraic cobordism spectrum} $\MGL$ to be the colimit of the Thom spectra of rank-zero K-theory classes over $\Sm_S$, i.e., 
\[
	\MGL = \colim_{\substack{X\in\Sm_S\\\rk\xi=0}}\Th_X(\xi) \in \MotSp_S.
\]
It is clear from this construction that $\MGL$ is canonically equipped with an $\E_\infty$-algebra structure. 
After imposing $\A^1$-invariance, our motivic spectrum $\MGL$ reduces to Voevodsky's.
In this sense, the latter theory should be regarded as the \emph{homotopy cobordism theory}, analogously to how algebraic K-theory reduces to homotopy K-theory after $\A^1$-localization \cite{Weibel, Cisinski}.

Beyond the construction, three main results on $\MGL$ are established. First of all, we prove that $\MGL$ is the universal homotopy commutative oriented ring spectrum, providing a non-$\A^1$-invariant refinement for the analogous claim in $\A^1$-homotopy theory \cite{Panin:2008}.

\begin{theorem}[Universality of $\MGL$, Theorem \ref{thm:universality}]\label{intro:MGL}
The algebraic cobordism spectrum $\MGL$ is the initial oriented object in $\CAlg(\h\MotSp_S)$.
\end{theorem}

Secondly, we prove that the algebraic K-groups can formally be recovered from the $\MGL$-cohomology groups, by imposing the multiplicative formal group law for Chern classes of line bundles. This is a non-$\A^1$-invariant refinement of \cite[Theorem 1.2]{Spitzweck-Bott}. Morally, it should also be the ``higher version'' of the Conner--Floyd theorem proved in \cite{annala-chern}. 

\begin{theorem}[Conner--Floyd isomorphism, Theorem \ref{thm:Conner-Floyd}]\label{intro:Conner-Floyd}
For any qcqs derived scheme $X$, there is an isomorphism of bigraded rings
\[
	\MGL^{**}(X)\otimes_{\mathrm{L}}\Z[\beta^{\pm 1}] \simeq \K^{**}(X),
\]
where $\L$ is the Lazard ring and $\K^{p,q}(X)=\K_{2q-p}(X)$.
\end{theorem}

The above result may be regarded as a sanity check for our construction, because it establishes a precise relationship between $\MGL$ and K-theory, the latter of which has a generally-accepted definition. 

Let $\PMGL$ be the \textit{periodic algebraic cobordism spectrum}, defined to be the colimit of the Thom spectra of \textit{all} K-theory classes over $\Sm_S$. Our third main result provides a concrete geometric model for $\PMGL$ as a motivic spectrum, refining the $\A^1$-homotopical Snaith theorem proved in \cite{GepnerSnaith}.

\begin{theorem}[Snaith theorem for $\PMGL$, Theorem \ref{thm:Snaith}]\label{intro:Snaith}
There is a canonical isomorphism
\[
	\PMGL \simeq \SigmaP^\infty\Vect_{\infty,+}[\beta^{-1}]
\]
in $\CAlg(\h\MotSp_S)$.
\end{theorem}

This result allows us to easily compute maps from $\PMGL$ to other motivic spectra. An analogous result was proven in \cite{AnnalaIwasa2} for algebraic K-theory. The advantage of the cobordism version is that algebraic cobordism has a much richer structure than algebraic K-theory, owing to the fact that K-theory is confined to the first chromatic level.
Hence, algebraic cobordism should be more useful than K-theory in studying, e.g., torsion in crystalline and syntomic cohomology and other infinite-chromatic-height phenomena.

In order to obtain the aforementioned results, we significantly advance the foundational understanding of non-$\A^1$-invariant motivic spectra. Our main insight is that $\otimes$-inverting the pointed projective line $\P^1$ leads, in a nontrivial fashion, to $\P^1$-homotopy invariance, and more generally to a twisted form thereof that we call \textit{$\P$-homotopy invariance}.

\begin{theorem}[$\P$-homotopy invariance, Theorem~\ref{thm:euler}]\label{intro:P-homotopy}
Let $\scr{E}$ be a finite locally free sheaf on $X\in\scr{P}(\Sm_S)$ and $\sigma\colon\scr{E}\to\scr{O}_X$ a linear map.
Then there is a canonical homotopy $h(\sigma)$ in $\MotSp_S$ between 
\[
	X \xrightarrow{\sigma} \V(\scr{E}) \subset \P(\scr{E}\oplus\scr{O}_X)
\]
and the zero section.
The homotopy $h(\sigma)$ is functorial in $(S,X,\scr{E},\sigma)$ and is the identity if $\sigma=0$.
\end{theorem}

Using $\P$-homotopy invariance in place of $\A^1$-homotopy invariance, we are able to prove several useful results for motivic spectra.

\begin{theorem}\label{intro:equivalences}
The following results hold in $\MotSp_S$.
\begin{enumerate}[beginpenalty=10000]
\item\label{item:intro-bass} {\upshape (Bass fundamental theorem, Proposition \ref{prop:fundamental})}
The canonical pointed map 
\[
	\P^1\to\Sigma\G_m
\]
admits a retraction.
\item\label{item:intro-euler} {\upshape (Euler class of $\scr{O}(1)$, Proposition \ref{prop:EulerP1})}
Let $s,i\colon\P^1\rightrightarrows\P_{\P^1}(\scr{O}(1)\oplus\scr{O})$ be the zero section and the inclusion of the fiber at infinity, respectively.
Then the two composites
\[
\begin{tikzcd}[cramped, sep=scriptsize]
	\P^1_+ \ar[r, shift left, "s"] \ar[r, shift right, "-i" swap] & \P_{\P^1}(\scr{O}(1)\oplus\scr{O})_+ \ar[r] & \P_{\P^1}(\scr{O}(1)\oplus\scr{O})/ \P_{\P^1}(\scr{O}(1))=\Th_{\P^1}(\scr O(1))
\end{tikzcd}
\]
are homotopic.
\item\label{item:intro-weight} {\upshape (Weighted $\A^1$-homotopy invariance, Corollary \ref{cor:A^1/G_m-homotopy})}
Let $\A^1/\G_m$ be the quotient stack with respect to a $\G_m$-action of nonzero weight.
Then the canonical map
\[
	\A^1/\G_m \to\rm B\G_m=\Pic
\]
is an equivalence.
\item\label{item:intro-Ainf} {\upshape (Infinite excision, Proposition \ref{prop:Stiefel})}
The open embedding
\[
	\A^\infty-0\into\A^\infty
\]
is an equivalence.
\item\label{item:intro-vect} {\upshape (Geometric model of the stack of vector bundles, Theorem \ref{thm:Gr-Vect})}
The canonical map
\[
	\Gr_n \to \Vect_n
\]
is an equivalence.
\end{enumerate}
\end{theorem}

In fact, \ref*{item:intro-weight} and \ref*{item:intro-Ainf} hold more generally; see the mentioned references and Lemma \ref{lem:V/G}.
The results listed in Theorem \ref{intro:equivalences} are all either well-known or obvious after $\A^1$-localization; \ref*{item:intro-bass}, \ref*{item:intro-weight}, \ref*{item:intro-Ainf} are obvious, \ref*{item:intro-euler} is contained in \cite[proof of Lemma 3.8]{PaninSmirnov}, and \ref*{item:intro-vect} is \cite[Section 4, Proposition 3.7]{MV}.
Without $\A^1$-localization, \ref*{item:intro-vect} was previously proved when restricted to oriented theories in \cite[Theorem 4.4.6]{AnnalaIwasa2} by adapting an argument from \cite{AnnalaIwasa}. The proof presented here is logically independent from the previous one, more general, and simpler.

The stable $\infty$-category $\MotSp_S$ is by definition the stabilization of $\MotSp^\un_S = \scr{P}_{\Zar,\ebu}(\Sm_S)_*[(\P^1)^{-1}]$. Combining the Bass fundamental theorem \ref*{item:intro-bass} with the stability theorem \cite[Theorem 2.4.5]{AnnalaIwasa2}, we obtain the following version of Bass delooping.

\begin{theorem}[Bass delooping, Corollary \ref{cor:delooping}]\label{intro:delooping}
The functor
\[
	\Omega^\infty \colon \MotSp_S \to \MotSp_S^\un 
\]
is fully faithful and its essential image consists exactly of the fundamental objects, i.e., those $E\in\MotSp^\un_S$ such that the canonical map $\P^1\otimes E\to \Sigma\G_m\otimes E$ admits a retraction.
\end{theorem}

\begin{remark}
	Theorem~\ref{intro:equivalences} applies in particular to any spectrum-valued cohomology theory satisfying the projective bundle formula (since this implies elementary blowup excision \cite[Lemma 3.3.5]{AnnalaIwasa2}), such as the syntomic cohomology of schemes defined in \cite{BhattLurie}.
	This shows that the computation of the syntomic cohomology of $X\times\BGL_n$ \cite[Theorem 9.3.1]{BhattLurie} and of the classifying stack of a parabolic subgroup \cite[Corollary 9.2.10]{BhattLurie} are ``formal'' consequences of the projective bundle formula \cite[Theorem 9.1.1]{BhattLurie}, and hence that the $p$-quasisyntomicity assumption in the statements of these results is not necessary.
\end{remark}

\begin{remark}
The cohomology groups $\MGL^{**}(X)$ are expected to provide \textit{higher algebraic cobordism groups}, extending the non-$\A^1$-invariant algebraic cobordism groups $\Omega^*(X)$ constructed in characteristic 0 in \cite{annala-cob} and over a general Noetherian base ring $A$ in \cite{AY, annala-chern, annala-pre-and-cob}. More precisely, for all quasi-projective derived $A$-schemes $X$, we expect canonical isomorphisms $\Omega^n(X) \simeq \MGL^{2n,n}(X)$ for all $n \in \Z$. We prove this with rational coefficients (Corollary~\ref{cor:QMGLOmega}), but establishing such a comparison with integer coefficients seems difficult, and is not pursued here. 
\end{remark}

\subsection*{Related work}

Other constructions of motivic homotopy categories without $\A^1$-invariance have been developed based on extensions of the category of schemes itself: one by Kelly and Miyazaki using modulus pairs \cite{KellyMiyazaki} and one by Binda, Park, and Østvær using log schemes \cite{BPOe}. 
Our construction is in some sense more naive, as it is simply a variant without $\A^1$-invariance of the $\P^1$-stable motivic homotopy category of Morel and Voevodsky. Since any scheme may be viewed as either a modulus pair or a log scheme, there are canonical functors from our $\infty$-category of motivic spectra to theirs. 

Binda, Park, and Østvær also prove similar results to ours in the logarithmic setting; they define in particular the logarithmic cobordism spectrum $\mathrm{logMGL}$ and prove its universality. Although their definition looks slightly different than our definition of $\MGL$, the universal properties imply that $\mathrm{logMGL}$ is the image of $\MGL$.

\subsection*{Conventions and notation}

We use the word ``anima'' for spaces/$\infty$-groupoids and we denote by $\Ani$ the $\infty$-category of anima. We write $\scr P(\scr C)$ for the $\infty$-category of presheaves of anima on $\scr C$. If $\tau$ is a Grothendieck topology on $\scr C$, $\scr P_\tau(\scr C)\subset\scr P(\scr C)$ is the full subcategory of $\tau$-sheaves.
We write $\Sp$ for the $\infty$-category of spectra and $\Sp(\scr C)$ for the $\infty$-category of spectrum objects in $\scr C$.
If $\scr C$ admits filtered colimits, we write $\scr C^\omega\subset\scr C$ for the full subcategory of compact objects.
We write $\h\scr C$ for the homotopy category of an $\infty$-category $\scr C$.

For a presentably symmetric monoidal $\infty$-category $\scr V$ and an object $X\in\scr V$, we write $\Sp_X(\scr V)=\scr V[X^{-1}]$ for the symmetric monoidal $\infty$-category of symmetric $X$-spectra \cite[Section 1]{AnnalaIwasa2} and $\Sigma^\infty_X\colon \scr V\to\Sp_X(\scr V)$ for the canonical functor.

We use the following indexing conventions for cohomology theories represented by $\P^1$-spectra. If $X$ is pointed, we write $\tilde E^n(X)=\pi_0\Map(\SigmaP^\infty X,\SigmaP^n E)$ and $\tilde E^{p,q}(X)=(\Sigma^{p-2q}\tilde E)^q(X)$. If $X$ is unpointed, we write $E^{*}(X)=\tilde E^{*}(X_+)$ and $E^{**}(X)=\tilde E^{**}(X_+)$.

A scheme is a derived scheme by default. Note that we often use hooked arrows $\into$ for immersions of derived schemes, even though these are not monomorphisms.
We write $\Sch_X$ for the $\infty$-category of $X$-schemes and $\Sm_X\subset \Sch_X$ for the full subcategory of smooth $X$-schemes. The superscript ``$\fp$'' means ``of finite presentation''.

We write $\Vect(X)$ for the anima of finite locally free sheaves over a scheme $X$, and $\Pic(X)=\Vect_1(X)$ for the subanima of invertible sheaves.
For a sheaf $\scr E\in\Vect(X)$, we denote by $\V(\scr E)=\Spec(\Sym\scr E)$ and $\P(\scr E)=\Proj(\Sym\scr E)$ the associated vector and projective bundles.

\subsection*{Acknowledgments} 
We thank Dustin Clausen for helpful discussions about weighted $\A^1$-homotopy invariance and infinite excision, Jacob Lurie for useful discussions and answering questions about his recent work with Bhargav Bhatt, and Vova Sosnilo for useful discussions and insights about weighted $\A^1$-homotopy invariance.

\section{Smooth blowup excision}
\label{sec:sbu}

Let $S$ be a derived scheme. We refer to \cite{KhanVCD} for the definition of the blowup of a derived scheme at a quasi-smooth closed subscheme.

\begin{definition}
	Let $\scr C$ be an $\infty$-category and $F\colon\Sm_S^\op\to\scr C$ a $\scr C$-valued presheaf.
	\begin{enumerate}[beginpenalty=10000]
		\item\label{item:sbu} We say that $F$ satisfies \emph{smooth blowup excision} if $F(\emptyset)$ is a final object of $\scr C$ and for every closed immersion $i\colon Z\into X$ in $\Sm_S$, $F$ sends the blowup square
		\[
		\begin{tikzcd}
			E \ar[hookrightarrow]{r} \ar{d} & \Bl_ZX \ar{d} \\
			Z \ar[hookrightarrow]{r}{i} & X
		\end{tikzcd}
		\]
		to a cartesian square.
		\item\label{item:ebu} A closed immersion $i\colon Z\into X$ is called \emph{elementary} if, Zariski-locally on $X$, it is the zero section of $\A^n_Z\sqcup Y$ for some $n\ge 0$ and some $Y$. We say that $F$ satisfies \emph{elementary blowup excision} if \ref*{item:sbu} holds whenever $i$ is elementary.
	\end{enumerate}
	We denote by $\scr P_\sbu(\Sm_S)\subset \scr P_\ebu(\Sm_S)$ the corresponding full subcategories of $\scr P(\Sm_S)$, and by $\L_\sbu$ and $\L_\ebu$ the corresponding localization functors, which preserve finite products.
\end{definition}

The above definition of elementary blowup excision is slightly less elementary than \cite[Definition 3.1.1]{AnnalaIwasa2}, but it is obviously equivalent for Zariski sheaves.
Note that if $F\colon\Sm_S^\op\to\scr C$ satisfies elementary blowup excision, then $F$ preserves finite products.
In particular, we have $\scr P_\ebu(\Sm_S)\subset \scr P_\Sigma(\Sm_S)$.
For Nisnevich sheaves of spectra, there is no difference between elementary and smooth blowup excision:

\begin{proposition}\label{prop:Nis-sbu}
	Suppose that $\scr C$ is stable and that $F\colon\Sm_S^\op\to\scr C$ satisfies Nisnevich descent and elementary blowup excision. Then $F$ satisfies smooth blowup excision.
\end{proposition}

\begin{proof}
	Let $i\colon Z\into X$ be a closed immersion in $\Sm_S$. Zariski-locally on $X$, there exist cartesian squares
	\[
	\begin{tikzcd}
		Z \ar[hookrightarrow]{d}[swap]{i} & Z \ar[equal]{l} \ar[equal]{r} \ar[hookrightarrow]{d}[swap]{t} & Z \ar[hookrightarrow]{d}[swap]{s} \\
		X & V \ar{l} \ar{r} & \A^n_Z\rlap,
	\end{tikzcd}
	\]
	where the horizontal maps are étale and $s$ is the zero section (the proof of \cite[Section~3, Lemma~2.28]{MV} works without change for derived schemes; alternatively, one can observe that the claim depends only on the underlying classical schemes, by the topological invariance of the étale site).
	Let $B\to X$ be the blowup of $X$ at $Z$, $E\subset B$ the exceptional divisor, $U=X-Z$ and $W=V-Z$. We then have a commutative diagram
	\[
	\begin{tikzcd}[row sep={25,between origins}, column sep={25,between origins}]
	      & E \ar[hookrightarrow]{rr}\ar{dd}\ar[equal]{dl} & & B\times_XV \ar{dd}\ar{dl}  \ar[hookleftarrow]{rr} & & W \ar[equal]{dd}\ar{dl} \\
	    E \ar[crossing over,hookrightarrow]{rr} \ar{dd} & & B \ar[crossing over,hookleftarrow]{rr} & & U  \\
	      & Z  \ar[hookrightarrow]{rr}[near start]{t} \ar[equal]{dl} & &  V \ar{dl} \ar[hookleftarrow]{rr} & &  W \ar{dl} \\
	    Z \ar[hookrightarrow]{rr}{i} && X \ar[from=uu,crossing over] \ar[hookleftarrow]{rr} && U\rlap. \ar[from=uu,equal,crossing over]
	\end{tikzcd}
	\]
	Since $F$ satisfies Nisnevich descent, it takes the top and bottom face on the right-hand side to cartesian squares, hence also the middle face in the diagram. Hence, on the left-hand side, we see that $F$ sends the blowup square for $i$ to a cartesian square if it sends the blowup square for $t$ to a cartesian square, and also conversely since $\scr C$ is stable.
	Since the same applies with $s$ instead of $i$, the claim follows.
\end{proof}

\begin{construction}[Cubes and total cofibers]
	\label{cst:cubes}
	For a finite set $I$, we denote by $\Box^I$ the poset of subsets of $I$.
	We let $\Cube\subset\Cat_\infty$ be the subcategory whose objects are the cubes $\Box^I$ and whose morphisms are the colimit-preserving functors.
	A morphism of cubes is called \emph{strict} if it sends non-initial objects to non-initial objects; we denote by $\Cube^s\subset\Cube$ the wide subcategory of strict morphisms. Thus:
	\[
	\Map_{\Cube}(\Box^I,\Box^J)\simeq\Map(I,\Box^J)\quad\text{and}\quad \Map_{\Cube^s}(\Box^I,\Box^J)\simeq\Map(I,\Box^J-\{\emptyset\}).
	\]
	The cartesian symmetric monoidal structure on $\Cat_\infty$ restricts to a symmetric monoidal structure on both $\Cube$ and $\Cube^s$ (which is cartesian on $\Cube$ but not on $\Cube^s$).
	Note that the $1$-cube $\Box^1$ is a final object of $\Cube^s$ and hence admits a unique structure of commutative monoid in $\Cube^s$, whose multiplication is $\cup$.
	
	Let $\scr C$ be an $\infty$-category. We denote by $\Cube(\scr C)\to \Cube$ the cartesian fibration classified by the functor
	\[
	\Cube^\op \to \Cat_\infty, \quad \Box^I\mapsto \Fun(\Box^{I,\op},\scr C),
	\]
	and by $\Cube^s(\scr C)\subset \Cube(\scr C)$ the wide subcategory given by the preimage of $\Cube^s\subset\Cube$. Thus:
	\begin{itemize}
		\item An object of $\Cube(\scr C)$ is a pair $(I,X)$ consisting of a finite set $I$ and an $I$-cube $X\colon \Box^{I,\op}\to \scr C$.
		\item A morphism $(I,X)\to (J,Y)$ in $\Cube(\scr C)$ is a map of cubes $\alpha\colon \Box^I\to \Box^J$ together with a natural transformation $X\to \alpha^* (Y)$. It is a morphism in $\Cube^s(\scr C)$ if and only if $\alpha$ is strict.
	\end{itemize}
	If $\scr C$ admits finite colimits, then each functor $\alpha^*$ admits a left adjoint $\alpha_!$, so that the cartesian fibration $\Cube(\scr C)\to \Cube$ is also cocartesian.

	Suppose now that $\scr C$ has a symmetric monoidal structure.
	Then the above functor $\Cube^\op \to \Cat_\infty$ is lax symmetric monoidal, so that $\Cube(\scr C)\to \Cube$ is a symmetric monoidal cartesian fibration: the tensor product $(I,X)\otimes (J,Y)$ is the $I\sqcup J$-cube $K\sqcup L\mapsto X(K)\otimes Y(L)$.
	If moreover $\scr C$ admits finite colimits that are preserved by the tensor product in each variable, then each functor $\alpha_!$ is symmetric monoidal, so that $\Cube(\scr C)\to \Cube$ is also a symmetric monoidal cocartesian fibration. Over $\Cube^s$, the total pushforward to the final object $\Box^1$ then gives a symmetric monoidal functor
	\begin{align*}
		\Cube^s(\scr C) &\to \Fun(\Box^{1,\op},\scr C)\simeq \Fun(\Delta^1,\scr C),\\
		(I,X) & \mapsto \left(\colim_{\emptyset\neq J\subset I} X(J) \to X(\emptyset)\right),
	\end{align*}
	where $\Fun(\Delta^1,\scr C)$ is equipped with the Day convolution (also known as the pushout-product).
	If $\scr C$ has a final object $*$, there is a further symmetric monoidal functor $\cofib\colon \Fun(\Delta^1,\scr C)\to\scr C_*$. Hence, we obtain a symmetric monoidal functor
	\[
	\operatorname{tcofib}\colon\Cube^s(\scr C) \to \scr C_*
	\]
	sending a cube in $\scr C$ to its total cofiber.
\end{construction}

\begin{definition}
	Let $S$ be a derived scheme and let $X$ be a smooth $S$-scheme.
	A \emph{relative strict normal crossings divisor} $\partial X$ on $X$ is the data of a finite set $I$ and of an $I$-cube 
	\[
	\Box^{I,\op} \to (\Sm_S)_{/X},\quad J\mapsto \partial_JX,
	\]
	such that $\partial_\emptyset X=X$ and:
	\begin{enumerate}
		\item the cube is strongly cartesian in $\Sch_S$ (i.e., it is right Kan extended from $\Box^{I,\op}_{\leq 1}$);
		\item for each subset $J\subset I$, the map $\partial_JX\to X$ in $\Sm_S$ is a closed immersion, which is everywhere of codimension $\lvert J\rvert$.
	\end{enumerate}
	We let $\Sm_S^\sncd$ denote the full subcategory of $\Cube^s(\Sm_S)$ spanned by the relative strict normal crossings divisors $(X,\partial X)$.
\end{definition}

By definition, a relative strict normal crossings divisor $\partial X$ on $X$ is uniquely determined by the smooth divisors $\partial_i X\to X$ with $i\in I$, called the \emph{smooth components} of $\partial X$.
Note that the symmetric monoidal structure on $\Cube^s(\Sm_S)$ restricts to $\Sm_S^\sncd$: if $\partial X$ has smooth components $\partial_iX$ and $\partial Y$ has smooth components $\partial_jY$, then the tensor product $(X,\partial X)\otimes (Y,\partial Y)$ is given by the relative strict normal crossings divisor on the smooth $S$-scheme $X\times_SY$ with smooth components $\partial_i X\times_SY$ and $X\times_S\partial_j Y$.
Given $(X,\partial X)\in \Sm_S^\sncd$, we will also denote by $\partial X$ the colimit of the punctured cube in $\scr P(\Sm_S)_{/X}$.
As explained in Construction~\ref{cst:cubes}, we then have symmetric monoidal functors
\begin{align*}
	\Sm_S^\sncd &\to \Fun(\Delta^1,\scr P(\Sm_S))\to \scr P(\Sm_S)_*,\\
	(X,\partial X) & \mapsto (\partial X\to X)\mapsto X/\partial X.
\end{align*}

\begin{remark}
	If $S$ is a classical scheme, one can show that the image of the presheaf $\partial X$ in $\scr P_\Sigma(\Sm_S)_{/X}$ depends only on the underlying Cartier divisor $\sum_{i\in I}\partial_i X$ on $X$, i.e., it is independent of the choice of smooth components of that divisor (indeed, such choices form a poset, and one can obtain a common refinement of any two choices by taking finite coproduct decompositions, see \cite[Proposition A.0.7]{NS}). We do not expect this to remain true for derived schemes, which is why we take the smooth components as part of the data of a relative strict normal crossings divisor.
\end{remark}

\begin{remark}[Functoriality of quasi-smooth blowups]
	An \emph{excess intersection square} is a commutative square of derived schemes
	\[
	\begin{tikzcd}
		Z' \ar[hook]{r}{i'} \ar{d}[swap]{g} & X' \ar{d}{f} \\
		Z \ar[hook]{r}{i} & X\rlap,
	\end{tikzcd}
	\]
	where $i$ and $i'$ are quasi-smooth closed immersions, such that the underlying square of topological spaces is cartesian and such that the induced map $g^*(\scr N_{i})\to \scr N_{i'}$ is surjective \cite{KhanExcess}. The blowup of quasi-smooth closed immersions is then functorial with respect to excess intersection squares.
	In classical geometry, given $i\colon Z\into X$ and $f\colon X'\to X$ as above, one often speaks of the \emph{strict transform} of $X'$ with respect to the blowup of $X$ at $Z$, which means the blowup of $X'$ at $Z\times_XX'$. In derived geometry, however, there are usually many ways of forming an excess intersection square, and it might not be the actual pullback $Z\times_XX'$ that is geometrically relevant; for example, given quasi-smooth closed immersions $Z\into Y\into X$, the relevant ``strict transform'' of $Y$ is often the blowup of $Y$ at $Z$ and not at $Z\times_XY$. In some of the geometric arguments in Section~\ref{sec:thom}, we will nevertheless use the term ``strict transform'' in cases where the intended excess intersection square is clear from the context. In fact, the geometric situations we will deal with are always classical in the sense that the universal example lives over a classical stack, where ``strict transform'' has its classical meaning.
\end{remark}

If $X$ is smooth over $S$ and $Z\subset X$ is a smooth closed subscheme, then $Z$ is Zariski-locally on $X$ the zero locus of a map $X\to \A^n$. We will refer to such a map as \emph{coordinates} along $Z$. If $\partial X$ is a relative strict normal crossings divisor on $X$ with smooth components $(\partial_i X)_{i\in I}$, then for every $J\subset I$ there are coordinates along $\partial_J X$ in which the divisors $\partial_i X$ with $i\in J$ are the coordinate hyperplanes.

By a \emph{smooth center} $Z$ on $(X,\partial X)$, we mean a closed immersion of $I$-cubes $(Z_J)_{J}\to (\partial_JX)_J$ in $\Sm_S$ such that, for each $J\subset I$, there are coordinates along $Z_J$ in which the divisors $\partial_i X$ with $i\in J$ are \emph{some} of the coordinate hyperplanes while $Z_\emptyset$ is the vanishing locus of \emph{some} subset of the coordinates. We will also write $Z$ for the underlying smooth closed subscheme $Z_\emptyset\into X$. Given a smooth center $Z$ on $(X,\partial X)$, each square
\[
\begin{tikzcd}
	Z_J \ar[hook]{r} \ar{d} & \partial_J X \ar{d} \\
	Z \ar[hook]{r} & X
\end{tikzcd}
\]
is an excess intersection square, and we call the blowup of $\partial_JX$ at $Z_J$ the \emph{strict transform} of $\partial_JX$. The strict transforms $\tilde\partial_iX$ of the components $\partial_i X$ together with the exceptional divisor $E$ then form a relative strict normal crossings divisor $\tilde\partial X\cup E$ on the blowup $\Bl_ZX$, with underlying finite set $I\sqcup\{e\}$.
Moreover, the intersection $\bigcap_{i\in J}\tilde\partial_i X$ is the strict transform $\tilde\partial_JX$ of $\partial_J X$, and $E\cap \tilde\partial_JX$ is the exceptional divisor of this blowup.

Let $Z$ be a smooth center on $(X,\partial X)$. Given a subset $K\subset I$, we say that $Z$ is \emph{contained in $\partial_K X$} if $Z_{J\cup K}\simto Z_J$ for all $J\subset I$. If $K$ is \emph{nonempty}, we obtain a morphism
\[
(\Bl_ZX,\tilde\partial X\cup E) \to (X,\partial X)
\]
in $\Sm_S^\sncd$, whose underlying strict morphism of cubes $\Box^{I\sqcup\{e\}}\to \Box^I$ sends $\{i\}$ to $\{i\}$ and $\{e\}$ to $K$.

\begin{proposition}\label{prop:blowup-square}
	Let $X$ be a smooth $S$-scheme, $\partial X$ a relative strict normal crossings divisor on $X$, and $Z$ a smooth center on $(X,\partial X)$ contained in $\partial_K X$ for some $K\neq\emptyset$.
	Then the square
	\[
	\begin{tikzcd}
		\tilde\partial X\cup E \ar{r} \ar{d} & \Bl_ZX \ar{d} \\
		\partial X \ar{r} & X
	\end{tikzcd}
	\]
	is cocartesian in $\scr P_{\sbu}(\Sm_S)$. If moreover each closed immersion $Z_J\into \partial_JX$ is elementary, then the square is cocartesian in $\scr P_\ebu(\Sm_S)$.
\end{proposition}

\begin{proof}
	We consider the following commutative diagram in $\scr P(\Sm_S)$:
	\[
	\begin{tikzcd}
		E\cap \tilde\partial X \ar{r} \ar{d} & \tilde\partial X \ar{d} & \\
		E \ar{r} \ar{d} & \tilde\partial X\cup E \ar{r} \ar{d} & \Bl_ZX \ar{d} \\
		Z\ar{r} & \partial X \ar{r} & X\rlap.
	\end{tikzcd}
	\]
	The upper square is a pushout square in $\scr P(\Sm_S)$, and the lower horizontal rectangle is a smooth blowup square. It remains to show that the left vertical rectangle is a pushout in $\scr P_{\sbu}(\Sm_S)$. This rectangle is the colimit in $\scr P(\Sm_S)$ of the squares
	\[
	\begin{tikzcd}
		E\cap \tilde\partial_J X \ar{r} \ar{d} & \tilde\partial_J X \ar{d} \\
		Z_J \ar{r} & \partial_J X\rlap,
	\end{tikzcd}
	\] 
	where $J$ ranges over the nonempty subsets of the underlying finite set of $\partial X$. Each of these squares is a smooth blowup square, which proves the claim.
\end{proof}

\section{Thom spaces}
\label{sec:thom}

Let $S$ be a derived scheme. For a smooth $S$-scheme $X$ (or more generally an arbitrary presheaf $X$ on $\Sm_S$) and a finite locally free sheaf $\scr E$ on $X$, we define the \emph{Thom space} of $\scr E$ by
\[
\Th_X(\scr E)=\P(\scr E\oplus\scr O)/\P(\scr E) \in\scr P(\Sm_S)_*.
\]

Let $\Vect^\epi(S)$ be the $\infty$-category of finite locally free sheaves on $S$ and epimorphisms (i.e., morphisms that are surjective on $\pi_0$). 
Recall that an epimorphism $\scr E\onto\scr F$ induces a linear embedding $\P(\scr F)\into \P(\scr E)$.
Consequently, the Thom space construction defines a functor
\[
\Vect^\epi(S)^\op\to \scr P(\Sm_S)_*, \quad \scr E \mapsto \Th_S(\scr E)=\P(\scr E\oplus \scr O)/\P(\scr E).
\]
This functor does not have a lax symmetric monoidal structure, as there is no natural map between the pointed presheaves
\[
\Th_S(\scr E\oplus \scr F)\quad\text{and}\quad \Th_S(\scr E)\wedge \Th_S(\scr F).
\]
Our goal in this section is to construct a symmetric monoidal structure on the composite functor
\[
\Vect^\epi(S)^\op\to \scr P(\Sm_S)_*\to \scr P_{\ebu}(\Sm_S)_*.
\]

For a finite locally free sheaf $\scr E$ on $S$, we will regard $\P(\scr E\oplus\scr O)$ as an object of $\Sm_S^\sncd$ with the smooth boundary divisor $\P(\scr E)$. The Thom space $\Th_S(\scr E)$ is thus the image of $\P(\scr E\oplus\scr O)$ by the symmetric monoidal functor
\[
\Sm_S^\sncd \to \scr P(\Sm_S)_*, \quad (X,\partial X)\mapsto X/\partial X,
\]
defined in Section~\ref{sec:sbu}.

Let now $\scr E=(\scr E_i)_{i\in I}$ be a finite collection of finite locally free sheaves on $S$, and let us contemplate the problem of relating $\prod_{i\in I}\P(\scr E_i\oplus\scr O)$ and $\P(\bigoplus_{i\in I}\scr E_i\oplus\scr O)$ in the $\infty$-category $\Sm_S^\sncd$.
To that end, we will construct an object $\B(\scr E)\in \Sm_S^\sncd$ and a zigzag
\begin{equation}\label{eqn:zigzag}
\prod_{i\in I}\P(\scr E_i\oplus \scr O)\xleftarrow{b_\Pi} \B(\scr E) \xrightarrow{b_\P} \P\left({\textstyle \bigoplus_{i\in I}\scr E_i}\oplus\scr O\right),
\end{equation}
such that both maps become isomorphisms in $\scr P_{\ebu}(\Sm_S)_*$.\footnote{For this construction, it would suffice to work with classical schemes, since the universal example is classical.}

Let us first consider the special case when $I=\{1,2\}$ and $\scr E_1=\scr E_2=\scr O$.
On the left-hand side of~\eqref{eqn:zigzag} we then have $\P^1\times\P^1$ with boundary divisor $(\infty\times\P^1)\cup(\P^1\times\infty)$, and on the right-hand side we have $\P^2$ with boundary divisor $\P^1$ at infinity. In this case, $\B(\scr O,\scr O)$ is the blowup of $\P^1\times\P^1$ at the point $(\infty,\infty)$, which can be identified with the blowup of $\P^2$ at the two points $[1\mathbin :0\mathbin :0]$ and $[0\mathbin :1\mathbin :0]$, and the boundary divisor $\partial \B(\scr O,\scr O)$ is the union of the three exceptional divisors.

We now explain the general construction. 
For a subset $J\subset I$, let
\[
Z_J = \P\left({\textstyle\bigoplus_{i\notin J}\scr E_i}\right)\subset \P\left({\textstyle \bigoplus_{i\in I}\scr E_i}\oplus\scr O\right).
\]
Thus, $Z_\emptyset$ is the boundary divisor, the $Z_i$'s are linear subspaces of $Z_\emptyset$ in a $\partial\Delta^I$ configuration, and for $J\neq\emptyset$ we have $Z_J=\bigcap_{i\in J}Z_i$.
For a subset $K\subset I$, let
\[
W_K=\prod_{i\in K}\P(\scr E_i)\times\prod_{i\notin K} \P(\scr E_i\oplus\scr O)\subset \prod_{i\in I}\P(\scr E_i\oplus \scr O).
\]
Thus, $\bigcup_{i\in I}W_i$ is the boundary divisor and $W_K=\bigcap_{i\in K}W_i$.

	We first describe the scheme $\B(\scr E)$ via its functor of points: a point of $\B(\scr E)$ is a family $(Y_J)_{J\subset I}$, where $Y_J\subset \P(\bigoplus_{i\in I}\scr E_i\oplus\scr O)$ is a linear subspace such that $Z_J$ is a hyperplane in $Y_J$, and such that if $K\subset J$ then $Y_K\supset Y_J$. 
	In other words, a point of $\B(\scr E)$ is a family of factorizations
	\[
	{\textstyle\bigoplus_{i\in I}\scr E_i}\oplus\scr O\onto \scr F_J\stackrel{\chi_J}\onto {\textstyle\bigoplus_{i\notin J}\scr E_i},
	\]
	such that the kernel $\scr L_J$ of $\chi_J$ is invertible, and such that if $J'\subset J$ then $\scr F_J$ is a quotient of $\scr F_{J'}$. 
	This is in turn equivalent to a family of invertible quotients $\phi_J\colon\bigoplus_{i\in J}\scr E_i\oplus\scr O\onto \scr L_J$ such that if $J'\subset J$ then the restriction of $\phi_J$ to $\bigoplus_{i\in J'}\scr E_i\oplus\scr O$ factors through $\phi_{J'}$.\footnote{Factoring through an epimorphism is merely a property when $S$ is a classical scheme, but it should of course be understood as the data of compatible factorizations when $S$ is a derived scheme.}
	The epimorphism $\phi_J$ defines a point of $\P(\bigoplus_{i\in J}\scr E_i\oplus\scr O)$, which can be thought of as the normal direction to $Z_J$ inside $Y_J$. We thus have canonical morphisms
	\[
	\pi_J\colon \B(\scr E)\to \P\left({\textstyle\bigoplus_{i\in J}\scr E_i}\oplus\scr O\right),
	\]
	which exhibit $\B(\scr E)$ as a closed subscheme of a product of $2^{\lvert I\rvert}$ projective bundles over $S$.
	Taking $J=I$ yields a morphism
	\[
	b_{\P}\colon \B(\scr E)\to \P\left({\textstyle\bigoplus_{i\in I}\scr E_i}\oplus\scr O\right),
	\]
	sending the family $(Y_J)_{J\subset I}$ to the point $Y_I$.
	Taking $J$ to be a singleton yields a morphism
	\[
	b_{\Pi}\colon \B(\scr E)\to \prod_{i\in I}\P(\scr E_i\oplus\scr O),
	\]
	sending the family $(Y_J)_{J\subset I}$ to $(Y_i)_{i\in I}$.

Next, we want to show that both $b_\P$ and $b_\Pi$ are sequences of smooth blowups. To see this, we will need the following description of strict transforms when blowing up zero loci of sections of vector bundles (the description of the blowup itself already appears in \cite[Theorem 122]{AnnalaThesis}):

\begin{lemma}[Blowing up zero loci]
	\label{lem:blowup}
	Let $X$ be a derived scheme, $\scr E$ a finite locally free sheaf on $X$, and $\sigma\colon \scr E\to\scr O$ a linear map.
	Then the blowup of $X$ at the zero locus of $\sigma$ classifies factorizations of $\sigma$ as
	\[
	\scr E\xrightarrow\phi \scr L\xrightarrow\tau\scr O,
	\]
	where $\scr L$ is invertible and $\phi$ is surjective.
	The exceptional divisor is then the zero locus of $\tau$.
	If moreover $\mu\colon\scr F\to\scr E$ is a universally injective morphism of finite locally free sheaves, then the strict transform of the zero locus of $\sigma\circ\mu$ is the zero locus of $\phi\circ\mu$.
\end{lemma}

\begin{proof}
	We use the description of the functor of points of the blowup from \cite{KhanVCD}: $\Bl_{Z(\sigma)}X$ classifies pairs $(\tau,f)$ consisting of a generalized Cartier divisor $\tau\colon\scr L\to \scr O$ and an $X$-morphism $f\colon Z(\tau)\to Z(\sigma)$ inducing an isomorphism of underlying classical schemes and such that the induced morphism of conormal sheaves $\scr E|_{Z(\tau)}\to \scr L|_{Z(\tau)}$ is surjective.
	We must show that this data is equivalent to that of a factorization of $\sigma$ as above.
	On the one hand, such a factorization induces a map $f\colon Z(\tau)\to Z(\sigma)$, which is an isomorphism on classical schemes by the surjectivity $\phi$, and the induced map of conormal sheaves is surjective since it is the restriction of $\phi$.
	Conversely, let $(\tau,f)$ be a pair as above. The map $f$ induces an $\scr O$-linear map 
	\[
	\phi\colon \scr E\to \fib(\scr O\to\scr O_{Z(\sigma)}) \xrightarrow{f^*} \fib(\scr O\to\scr O_{Z(\tau)})=\scr L
	\]
	over $\scr O$, whose restriction to $Z(\tau)$ is the morphism of conormal sheaves induced by $f$. It remains to observe that $\phi$ is surjective: it is surjective over the points of $Z(\tau)=Z(\sigma)$ by assumption; over the complement, $\tau$ is an isomorphism and $\sigma$ is surjective, so that $\phi$ is also surjective.\footnote{This argument shows that, in the description of the blowup in \cite[Remark 4.1.3(ii)]{KhanVCD}, we can replace ``isomorphism of classical schemes'' by ``isomorphism of reduced schemes''.}
	
	In the final statement, the assumption that $\mu$ is universally injective guarantees that $Z(\sigma)$ is a quasi-smooth closed subscheme of $Z(\sigma\circ \mu)$, namely the zero locus of the induced map $\bar\sigma\colon \operatorname{coker}\mu \to\scr O$. It is then clear from the above description that $Z(\phi\circ\mu)$ is the blowup of $Z(\sigma\circ\mu)$ at the zero locus of $\bar\sigma$.
\end{proof}

By definition, $\B(\scr E)$ parametrizes $I$-cubes of invertible sheaves $(\scr L_J)_{J\subset I}$ with a surjective map from the $I$-cube $(\bigoplus_{i\in J}\scr E_i\oplus\scr O)_{J\subset I}$. We let $\B_{\geq r}(\scr E)$ be the functor parametrizing the same data but with $\lvert J\rvert\geq r$. The morphism $b_\P$ can then be factored as
\[
\B(\scr E) = \B_{\geq 0}(\scr E)=\B_{\geq 1}(\scr E) \to \B_{\geq 2}(\scr E)\to\dotsb\to \B_{\geq \lvert I\rvert}(\scr E)=\P\left({\textstyle\bigoplus_{i\in I}\scr E_i}\oplus\scr O\right).
\]

\begin{proposition}\label{prop:blowup-sequence1}
	For each $0\leq r\leq \lvert I\rvert-1$, the map $\B_{\geq r}(\scr E)\to\B_{\geq r+1}(\scr E)$ is a blowup with center $\coprod_{\lvert J\rvert=r}\tilde Z_J$, where $\tilde Z_J$ is the strict transform of $Z_J$. 
	For $J\subset I$ with $\lvert J\rvert=r$, define
	\[\scr L_{>J}= \lim_{J\subsetneqq J'} \scr L_{J'}\]
	in the stable $\infty$-category of quasi-coherent sheaves on $\B_{\geq r+1}(\scr E)$. Then $\scr L_{>J}$ is an invertible sheaf, locally isomorphic to $\scr L_{J\cup\{i\}}$ with $i\in I-J$, and $\tilde Z_J$ is the zero locus of the map
	\[
	\bigoplus_{i\in J}\scr E_i\oplus\scr O\to \scr L_{>J}
	\]
	induced by the maps $\phi_{J'}$ for $J\subsetneqq J'$.
\end{proposition}

\begin{proof}
	Assuming the given description of $\tilde Z_J$ for $\lvert J\rvert=r$, Lemma~\ref{lem:blowup} says that blowing up $\tilde Z_J$ in $\B_{\geq r+1}(\scr E)$ adds the data of a factorization
	\[
	\bigoplus_{i\in J}\scr E_i\oplus\scr O\overset{\phi_J}\onto \scr L_J\to \scr L_{>J}.
	\]
	We therefore obtain exactly $\B_{\geq r}(\scr E)$ by blowing up all $\tilde Z_J$'s with $\lvert J\rvert=r$, as claimed.
	
	Consider the right Kan extension $\tilde{\scr L}$ of the diagram of sheaves $\scr L$ on $\B_{\geq r+1}(\scr E)$ from the poset of subsets $J\subset I$ of size $\geq r+1$ to the poset of all subsets of $I$, so that $\tilde{\scr L}_J=\scr L_{>J}$ when $\lvert J\rvert=r$. We will show more generally that, for any $J\subset I$ of size $\leq r$, $\tilde{\scr L}_J$ is an invertible sheaf such that the strict transform $\tilde Z_J$ is the zero locus of the map
	\[
	\bigoplus_{i\in J}\scr E_i\oplus\scr O\to \tilde{\scr L}_{J}.
	\]
	(This is in fact true for all $J\subset I$, but trivial if $\lvert J\rvert\geq r+1$.)
	We assume inductively that $\B_{\geq r+1}(\scr E)$ is a sequence of blowups as claimed, and that the strict transforms of the $Z_J$'s up to $\B_{\geq r+2}(\scr E)$ have the above description. 
	
	For every $J\subsetneqq I$ with $\lvert J\rvert\geq r+1$, let $E_{J}$ be the Cartier divisor on $\B_{\geq r+1}(\scr E)$ which is the zero locus of $\scr L_J\to\scr L_{>J}$, i.e., the preimage of the exceptional divisor over the strict transform $\tilde Z_{J}\subset \B_{\geq\lvert J\rvert+1}(\scr E)$. 
	Let $U_i\subset \B_{\geq r+1}(\scr E)$ be the open complement of $\bigcup_{i\notin J}E_{J}$.
	Note that $E_J\cap E_{J'}=\emptyset$ whenever $J$ and $J'$ are not contained in one another, since then the strict transforms of $Z_J$ and $Z_{J'}$ became disjoint in the blowup $\B_{\geq \lvert J\cup J'\rvert}(\scr E)$.
	It follows that $\B_{\geq r+1}(\scr E)$ is covered by any $\lvert I\rvert-r$ of the open subsets $U_i$.
	For any $J\subset I$, we deduce by descending induction on $\lvert J\rvert$ that the open subsets $U_R=\bigcap_{i\in R}U_i$ with 
	$\lvert J\cup R\rvert\geq r+1$
	form an open covering of $\B_{\geq r+1}(\scr E)$.
	
	For $J\subsetneqq I$ with $\lvert J\rvert\geq r+1$, let $\scr O(-E_J)=\scr L_J\otimes\scr L_{>J}^{-1}$ and let $\sigma_J\colon\scr O(-E_J)\to\scr O$ be the canonical map, whose zero locus is $E_J$. Taking the determinant of the cartesian cube defining $\scr L_{>J}$, we find
	\[
	\scr O(-E_J)\simeq\bigotimes_{J\subset J'}\scr L_{J'}^{(-1)^{\lvert J'-J\rvert}}\quad\text{and hence}\quad \bigotimes_{J\subset J'\neq I}\scr O(-E_{J'})\simeq\scr L_J\otimes\scr L_I^{-1}.
	\]
	Under this isomorphism, the map $\scr L_J\to\scr L_{J\cup\{i\}}$ corresponds to the tensor product of the maps $\sigma_{J'}$ with $i\notin J'$.
	In other words, the zero locus of $\scr L_{J}\to \scr L_{J\cup\{i\}}$ is the union of the divisors $E_{J'}$ with $J\subset J'$ and $i\notin J'$. 
	Hence, for any $R\subset I$, the map $\scr L_{J}\to \scr L_{J\cup R}$ is an isomorphism over $U_R$.
	Passing to the right Kan extension, this implies that the map $\tilde{\scr L}_{J}\to \tilde{\scr L}_{J\cup R}$ is an isomorphism over $U_R$ for \emph{all} subsets $J,R\subset I$: this follows from the fact that the functor 
	\begin{align*}
		\{J'\suchthat J\subset J'\text{ and }\lvert J'\rvert\geq r+1\} &\to \{J'\suchthat J\cup R\subset J'\text{ and }\lvert J'\rvert\geq r+1\},\\
		J'&\mapsto J'\cup R,
	\end{align*}
	is coinitial, since it is left adjoint to the inclusion.
	Since the $U_R$'s with $\lvert J\cup R\rvert\geq  r+1$ form an open covering of $\B_{\geq r+1}(\scr E)$, we see that $\tilde{\scr L}_J$ is locally isomorphic to the invertible sheaves $\scr L_{J\cup R}$.
	
	Fix $J\subset I$ of size $\leq r$ and let $R\subset I$ be such that $\lvert J\cup R\rvert= r+1$. It remains to show that $\tilde Z_J\cap U_R$ is the zero locus of $\bigoplus_{i\in J}\scr E_i\oplus\scr O\to \tilde{\scr L}_{J}$ over $U_R$, as these $U_R$'s cover $\B_{\geq r+1}(\scr E)$. Since $\tilde{\scr L}_J\simeq\scr L_{J\cup R}$ over $U_R$, this is the same as the zero locus of
	\[
	\bigoplus_{i\in J}\scr E_i\oplus\scr O\into \bigoplus_{i\in J\cup R}\scr E_i\oplus\scr O \xrightarrow{\phi_{J\cup R}} \scr L_{J\cup R}.
	\]
	By the description of strict transforms from Lemma~\ref{lem:blowup} and the induction hypothesis, this locus is exactly the strict transform of $Z_J$ in the blowup of $\B_{\geq r+2}(\scr E)$ at $\tilde Z_{J\cup R}$. 
	To conclude, we observe that the other exceptional divisors of the blowup $\B_{\geq r+1}(\scr E)\to\B_{\geq r+2}(\scr E)$, i.e., the divisors $E_S$ with $\lvert S\rvert=r+1$ and $S\neq J\cup R$, do not intersect $\tilde Z_J\cap U_R$.
	If $J\not\subset S$, then $J\cup S$ has size $\geq r+2$ and hence the strict transforms of $Z_{J}$ and $Z_{S}$ in $\B_{\geq r+2}(\scr E)$ are disjoint. If $J\subset S$ but $S\neq J\cup R$, then $R\not\subset S$ and hence $E_{S}\cap U_R=\emptyset$ by definition of $U_R$. 
\end{proof}

In order to see that $b_\Pi$ is analogously a sequence of blowups at the strict transforms of the subschemes $W_K$, we need a dual description of $\B(\scr E)$. 
Let $\B^\vee(\scr E)$ be the functor parametrizing families of invertible quotients $\scr E_i\oplus\scr O\onto\scr L_i$ for $i\in I$, together with a compatible family of universally injective maps $\psi_K\colon \scr M_K\into \bigoplus_{i\in K}\scr L_i$ for all nonempty subsets $K\subset I$, where $\scr M_K$ is an invertible sheaf.
By a ``compatible family'' we mean that for any nonempty $K'\subset K$, the composition of $\psi_K$ with the projection $\bigoplus_{i\in K}\scr L_i\onto \bigoplus_{i\in K'}\scr L_i$ factors through $\psi_{K'}$, and moreover that $\scr O\to\bigoplus_{i\in I}\scr L_i$ factors through $\psi_I$; in other words, the maps $\psi_K$ form a morphism of punctured $I$-cubes under $\scr O$.
Let further $\B^\vee_{> r}(\scr E)$ be the functor parametrizing such families with $\lvert K\rvert>r$.
We then have a sequence of forgetful maps
\[
\B^\vee(\scr E)=\B^\vee_{>0}(\scr E)=\B^\vee_{>1}(\scr E) \to \dotsb\to \B^\vee_{>\lvert I\rvert-1}(\scr E)\to\B^\vee_{>\lvert I\rvert}(\scr E)=\prod_{i\in I}\P(\scr E_i\oplus\scr O).
\]

\begin{proposition}\label{prop:blowup-sequence2}
	For each $1\leq r\leq \lvert I\rvert$, the map $\B^\vee_{> r-1}(\scr E)\to\B^\vee_{> r}(\scr E)$ is a blowup with center $\coprod_{\lvert K\rvert=r}\tilde W_K$, where $\tilde W_K$ is the strict transform of $W_K$.
	For $K\subset I$ with $\lvert K\rvert=r$, define
	\[
	\scr M_{>K}=
	\begin{cases}
			\scr O, & \text{if $K=I$,} \\
			\colim_{K\subsetneqq K'}\scr M_{K'}, & \text{otherwise,}
		\end{cases}
	\]
	in the stable $\infty$-category of quasi-coherent sheaves on $\B^\vee_{> r}(\scr E)$.
	Then $\scr M_{>K}$ is an invertible sheaf, locally isomorphic to $\scr M_{K\cup\{i\}}$ with $i\in I-K$, and $\tilde W_K$ is the zero locus of the map
	\[
	\scr M_{>K}\to\bigoplus_{i\in K}\scr L_i
	\]
	induced by the maps $\psi_{K'}$  for $K\subsetneqq K'$.
\end{proposition}

\begin{proof}
	The proof is similar to that of Proposition~\ref{prop:blowup-sequence1}:
	the colimit defining $\scr M_{>K}$ (and more generally the left Kan extension of $\scr M$ to a punctured $I$-cube under $\scr O$) is locally trivial, and blowing up the zero locus of $\scr M_{>K}\to\bigoplus_{i\in K}\scr L_i$ in $\B^\vee_{> r}(\scr E)$ adds the data of a factorization
	\[
	\scr M_{>K}\to\scr M_K\overset{\psi_K}\into\bigoplus_{i\in K}\scr L_i,
	\]
	leading to $\B^\vee_{> r-1}(\scr E)$.
\end{proof}

\begin{lemma}[Stable duality for punctured cubes]
	\label{lem:poset-verdier}
	Let $I$ be a finite set, let $P$ be the poset such that $P^\triangleright=(\Delta^1)^I$, and let $\scr C$ be a stable $\infty$-category. Then there is a canonical isomorphism
	\begin{align*}
		\Fun(P,\scr C)&\simto \Fun(P^\op,\scr C), \\
		F&\mapsto	\left(p\mapsto \colim_{q\in P_{p/}} F(q)\right),
	\end{align*}
	with inverse
	\begin{align*}
		\Fun(P^\op,\scr C)&\simto \Fun(P,\scr C), \\
		G&\mapsto \left(p\mapsto \lim_{q\in (P_{p/})^\op} G(q)\right).
	\end{align*}
\end{lemma}

\begin{proof}
	Let $\Cat_\infty^\st$ be the symmetric monoidal $\infty$-category of small stable $\infty$-categories, whose unit is the $\infty$-category $\Sp^\fin$ of finite spectra.
	Let $K$ be a finite $\infty$-category, all of whose mapping anima are also finite (e.g., a finite poset). Then the stable $\infty$-category $\Fun(K,\Sp^\fin)$ is dualizable in $\Cat_\infty^\st$ with dual $\Fun(K,\Sp^\fin)^\op=\Fun(K^\op,\Sp^\fin)$ (see for example \cite[Section 4.3]{HSS}); the coevaluation is given by
	\[
	\coev\colon \Sp^\fin \to \Fun(K^\op,\Sp^\fin)\otimes \Fun(K,\Sp^\fin)\simeq \Fun(K^\op\times K,\Sp^\fin),\quad \1\mapsto \Sigma^\infty_+\Map_K(\ph,\ph).
	\]
	Consider the symmetric pairing
	\[
	\lambda\colon \Fun(K,\Sp^\fin) \otimes \Fun(K,\Sp^\fin)\xrightarrow{\otimes} \Fun(K,\Sp^\fin) \xrightarrow{\colim} \Sp^\fin.
	\]
	By duality, it induces an exact functor
	\[
	\mathbb{D}=(\id\otimes \lambda)\circ (\coev\otimes\id)\colon \Fun(K,\Sp^\fin)\to\Fun(K^\op,\Sp^\fin),
	\]
	which is explicitly given by the formula 
	\[
	\mathbb D(F)(x)=\colim_{y\in K}\Map_K(x,y)\otimes F(y)=\colim_{y\in K_{x/}} F(y).
	\]
	Here, the second equality is obtained by decomposing the colimit over $K_{x/}$ along the cocartesian fibration $K_{x/}\to K$ with fibers $\Map_K(x,\ph)$.
	
	Let us further assume that $(\Fun(K,\Sp^\fin),\colim)$ is a Frobenius algebra in $\Cat_\infty^\st$, i.e., that the above pairing $\lambda$ is nondegenerate. 
	Then $\mathbb D$ is an isomorphism satisfying $\mathbb D=\mathbb D^\vee$.
	For a morphism $f$ between dualizable objects in $\Cat_\infty^\st$, the dual morphism $f^\vee$ is left adjoint to $f^\op$, hence is equal to $(f^\op)^{-1}$ when $f$ is an isomorphism. We therefore have $\mathbb D^{-1}=\mathbb D^\op$.
	Tensoring $\mathbb D$ with any $\scr C\in\Cat_\infty^\st$, we obtain an isomorphism
 	\[
 	\mathbb D_{\scr C}\colon \Fun(K,\scr C)\simto\Fun(K^\op,\scr C)
 	\]
	such that $\mathbb D_{\scr C}^{-1}=\mathbb D_{\scr C^\op}^\op$. Thus, for $F\in\Fun(K,\scr C)$ and $G\in\Fun(K^\op,\scr C)$, we have the desired formulas
	\begin{align*}
		\mathbb D_\scr C(F)(x)&=\colim_{y\in K_{x/}} F(y),\\
		\mathbb D_\scr C^{-1}(G)(x)&=\lim_{y\in (K_{x/})^\op} G(y).
	\end{align*}
	
	It remains to show that $(\Fun(P,\Sp^\fin),\colim)$ is a Frobenius algebra in $\Cat_\infty^\st$. Passage to opposite categories is a symmetric monoidal automorphism of $\Cat_\infty^\st$, sending the pair $(\Fun(P,\Sp^\fin),\colim)$ to the pair $(\Fun(P^\op,\Sp^\fin),\lim)$. 
	But the latter is a Frobenius algebra by \cite[Example 1.10]{AokiVerdier}, since $P^\op$ is the face poset of a simplex.
\end{proof}

\begin{proposition}\label{prop:blowup-duality}
	There is a canonical isomorphism $\B(\scr E)\simeq \B^\vee(\scr E)$ over $\prod_{i\in I}\P(\scr E_i\oplus\scr O)$. In particular, both maps
	\[
	\prod_{i\in I}\P(\scr E_i\oplus \scr O)\xleftarrow{b_\Pi} \B(\scr E) \xrightarrow{b_\P} \P\left({\textstyle \bigoplus_{i\in I}\scr E_i}\oplus\scr O\right)
	\]
	are sequences of smooth blowups, as described in Propositions \ref{prop:blowup-sequence1} and~\ref{prop:blowup-sequence2}
\end{proposition}

\begin{proof}
	Given a point $(\scr L_J,\phi_J)_J$ of $\B(\scr E)$, we set
	\begin{equation*}\label{eqn:LtoM}
	\scr M_K=\lim_{\emptyset\neq J\subset K} \scr L_J
	\end{equation*}
	for $K$ nonempty, where the limit is computed in the stable $\infty$-category of quasi-coherent sheaves.
	The sheaf $\scr M_K$ is then locally isomorphic to $\scr L_i$ with $i\in K$. 
	Indeed, using the notation from Proposition~\ref{prop:blowup-sequence1}, we have $\scr M_I=\scr L_{>\emptyset}$, and we can reduce to the case $K=I$ using the forgetful map $\B(\scr E)\to\B(\scr E|_K)$.
	Moreover, we have a compatible family of maps $\psi_K\colon \scr M_K\to \bigoplus_{i\in K}\scr L_i$, which are universally injective (since they locally identify $\scr M_K$ with some $\scr L_i$).
	This defines a map $\B(\scr E)\to \B^\vee(\scr E)$.
	
	Conversely, given a point $(\scr M_K,\psi_K)_K$ of $\B^\vee(\scr E)$, we set $\scr L_\emptyset=\scr O$ and
	\begin{equation*}\label{eqn:MtoL}
	\scr L_J=\colim_{\emptyset\neq K\subset J}\scr M_K
	\end{equation*}
	for $J$ nonempty, where the colimit is computed in the stable $\infty$-category of quasi-coherent sheaves.
	Since $(\scr M_K)_K$ is a diagram under $\scr O$, we obtain a factorization
	\[
	\begin{tikzcd}
		\bigoplus_{i\in J}(\scr E_i\oplus\scr O) \ar[twoheadrightarrow]{r} \ar[twoheadrightarrow]{d}[swap]{\nabla} & \bigoplus_{i\in J}\scr L_i \ar{d} \\
		\bigoplus_{i\in J}\scr E_i\oplus\scr O \ar[dashed]{r}{\phi_J} & \scr L_J\rlap.
	\end{tikzcd}
	\]
	Using Proposition~\ref{prop:blowup-sequence2} and the forgetful map $\B^\vee(\scr E)\to\B^\vee(\scr E|_J)$, we see as above that $\scr L_J$ is locally isomorphic to $\scr L_i$ with $i\in J$, so that the right vertical map and hence $\phi_J$ are surjective.
	This defines a map $\B^\vee(\scr E)\to \B(\scr E)$.
	
	The fact that these constructions are inverse to one another follows from Lemma~\ref{lem:poset-verdier}.
\end{proof}

We  now define a relative strict normal crossings divisor $\partial\B(\scr E)$ on $\B(\scr E)$ as follows.
For any $J\subsetneqq I$, let $E_J\subset \B(\scr E)$ be the zero locus of $\scr L_J\to\scr L_{>J}$. Dually, for any nonempty $K\subset I$, let $E_K^\vee\subset \B^\vee(\scr E)$ be the zero locus of $\scr M_{>K}\to\scr M_{K}$. Under the isomorphism $\B(\scr E)\simeq \B^\vee(\scr E)$ of Proposition~\ref{prop:blowup-duality}, we have $E_J=E^\vee_{I-J}$. We then let $\partial\B(\scr E)$ consist of the $2^{\lvert I\rvert}-1$ smooth components $E_J$, or equivalently of the $2^{\lvert I\rvert}-1$ smooth components $E^\vee_K$.
The morphisms $b_\P$ and $b_\Pi$ are then morphisms in $\Sm_S^\sncd$, and it follows from repeated applications of Proposition~\ref{prop:blowup-square} that they both induce isomorphisms in $\scr P_{\ebu}(\Sm_S)_*$ after collapsing the boundary divisors (in fact, they both induce pushout squares in $\scr P_{\ebu}(\Sm_S)$ prior to quotienting).
This completes the construction of the zigzag~\eqref{eqn:zigzag}.
 In particular, the pointed presheaves
\[
\bigwedge_{i\in I}\Th_S(\scr E_i)\quad\text{and}\quad \Th_S\left({\textstyle\bigoplus_{i\in I}\scr E_i}\right)
\]
become isomorphic in $\scr P_{\ebu}(\Sm_S)_*$.

\begin{corollary}\label{cor:Thom-invertible}
	Let $\scr E$ be a finite locally free sheaf on $S$. Then $\Th_S(\scr E)$ is invertible in the symmetric monoidal $\infty$-category $\Sp_{\P^1}(\scr P_{\Zar,\ebu}(\Sm_S)_*)$.
\end{corollary}

\begin{proof}
	The assignment $S\mapsto \Sp_{\P^1}(\scr P_{\Zar,\ebu}(\Sm_S)_*)$ is a Zariski sheaf of symmetric monoidal $\infty$-categories. Since the functor $\Pic\colon \CAlg(\Cat_\infty)\to \Sp_{\geq 0}$ preserves limits, the assertion that $\Th_S(\scr E)$ is invertible is Zariski-local on $S$. We may thus assume that $\scr E=\scr O^n$. In this case, the above construction gives a zigzag of isomorphisms between $\Th_S(\scr O^n)$ and $\Th_S(\scr O)^{\otimes n}=(\P^1)^{\otimes n}$, which is invertible.
\end{proof}

The construction $\scr E\mapsto\B(\scr E)$ is evidently functorial in the family $\scr E\in\Vect^\epi(S)^I$ as well as in the base scheme $S$.
We now examine its functoriality in the indexing set $I$. For a morphism of finite sets $\alpha\colon I\to J$, let us consider more generally
\[
\B(\scr E,\alpha)=\prod_{j\in J}\B(\scr E|_{\alpha^{-1}(j)}) \in \Sm_S^\sncd.
\]
The points of $\B(\scr E,\alpha)$ are thus families of invertible quotients $\phi_A\colon\bigoplus_{i\in A}\scr E_i\oplus\scr O\onto\scr L_A$ with $A\subset \alpha^{-1}(j)$ and $j\in J$, such that if $A'\subset A$ then the restriction of $\phi_A$ to the domain of $\phi_{A'}$ factors through $\phi_{A'}$.
Consider a morphism
\[
\begin{tikzcd}
	I \ar{r}{\gamma} \ar{d}[swap]{\alpha} & K \ar{d}{\beta} \\
	J & L \ar{l}{\delta}
\end{tikzcd}
\]
from $\alpha$ to $\beta$ in the twisted arrow category $\Tw(\Fin)$. 
For every $l\in L$ and $B\subset \beta^{-1}(l)$, we then have $\gamma^{-1}(B)\subset \alpha^{-1}(\delta(j))$.
We therefore have a well-defined morphism in $\Sm_S^\sncd$:
\begin{equation}\label{eqn:B-functoriality}
\B(\scr E,\alpha)\to\B(\gamma_\oplus\scr E,\beta),\quad 
\left((\phi_A)_{A\subset\alpha^{-1}(j)}\right)_{j\in J}\mapsto \left((\phi_{\gamma^{-1}(B)})_{B\subset \beta^{-1}(l)}\right)_{l\in L}.
\end{equation}
The span~\eqref{eqn:zigzag} is a special case of this functoriality, applied to the span in $\Tw(\Fin)$
\[
(I\to I)\leftarrow (I\to *)\rightarrow (*\to *).
\]
Since $b_\Pi$ and $b_\P$ are $\L_\ebu$-equivalences, it follows from 2-out-of-3 that all maps~\eqref{eqn:B-functoriality} are $\L_\ebu$-equivalences.
In particular, for any iterated decomposition $I=I_0\to\dotsb\to I_n=*$ of the finite set $I$, we have a refinement of~\eqref{eqn:zigzag} to a diagram of $\L_\ebu$-equivalences $\Tw(\Delta^n)\to\Sm_S^\sncd$. For example, for $I\stackrel\alpha\to J\to *$ we get the diagram
\begin{equation}\label{eqn:2-zigzag}
\begin{tikzcd}[column sep={65,between origins}]
	& & \B(\scr E) \ar{dl} \ar{dr} & & \\
	& \prod_{j\in J}\B(\scr E|_{\alpha^{-1}(j)}) \ar{dl} \ar{dr} & & \B(\alpha_\oplus\scr E) \ar{dl} \ar{dr} & \\
	\prod_{i\in I}\P(\scr E_i\oplus\scr O) & & \prod_{j\in J}\P\bigl(\bigoplus_{i\in\alpha^{-1}(j)}\scr E_i\oplus\scr O\bigr) & & \P\bigl(\bigoplus_{i\in I}\scr E_i\oplus\scr O\bigr)\rlap.
\end{tikzcd}
\end{equation}

We now explain how to equip the functor
\[
\Th_S\colon \Vect^\epi(S)^\op\to \scr P_{\ebu}(\Sm_S)_*
\]
with a symmetric monoidal structure, which is moreover natural in $S$.\footnote{For our applications in this paper, we only need the functor $\Th_S$ on the maximal subgroupoid $\Vect(S)\subset\Vect^\epi(S)$, but this does not simplify the construction.}
Both $S\mapsto\Vect^\epi(S)^\op$ and $S\mapsto\scr P_{\ebu}(\Sm_S)_*$ are functors from $\Sch^\op$ to $\CAlg(\Cat_\infty)$. We let $\Vect^{\epi,\op,\otimes}$ and $\scr P_{\ebu}(\Sm)_*^\otimes$ denote the total spaces of the corresponding cocartesian fibrations over $\Sch^\op\times\Fin_*$. 
Our goal is thus to construct a functor 
\begin{equation}\label{eqn:Thom-cocart}
	\Th\colon \Vect^{\epi,\op,\otimes}\to \scr P_{\ebu}(\Sm)_*^\otimes
\end{equation}
over $\Sch^\op\times\Fin_*$, whose value on a triple $(S,I_+,(\scr E_i)_{i\in I})$ is $(S,I_+,(\Th_S(\scr E_i))_{i\in I})$.

To give an idea of what is involved, let us consider the desired effect of the functor~\eqref{eqn:Thom-cocart} on morphisms.
A morphism from $(S,I_+,(\scr E_i)_{i\in I})$ to $(T,J_+,(\scr F_j)_{j\in J})$ in $\Vect^{\epi,\op,\otimes}$ consists of
\[
S\overset{f}\leftarrow T,\quad I_+\overset{\alpha}\rightarrow J_+,\quad \left({\textstyle\bigoplus_{\alpha(i)=j} f^*(\scr E_i)}\overset{\phi_j}\twoheadleftarrow \scr F_j\right)_{j\in J}.
\]
We assign to it the $J$-indexed family of morphism $\bigwedge_{\alpha(i)=j} f^*(\Th_S(\scr E_i))\to \Th_T(\scr F_j)$, given by precomposing $\Th(\phi_j)\colon \Th_T(\bigoplus_{\alpha(i)=j} f^*(\scr E_i))\to \Th_T(\scr F_j)$ with $f^*$ of the span of $\L_{\ebu}$-equivalences
\[
\bigwedge_{\alpha(i)=j}\Th_S(\scr E_i) \leftarrow \B((\scr E_i)_{\alpha(i)=j})/\partial \B \rightarrow \Th_S\left({\textstyle\bigoplus_{\alpha(i)=j} \scr E_i}\right).
\]
To explain the construction of~\eqref{eqn:Thom-cocart} in full, we need a brief categorical digression.

Given an $\infty$-category $\scr E$ with two classes of morphisms $\scr L$ and $\scr R$ closed under composition, we denote by $\Lambda(\scr E,\scr L,\scr R)$ the simplicial anima whose $n$-simplices are diagrams $\mathrm{Tw}(\Delta^n)\to \scr E$ sending $(\Delta^n)^\op$ to $\scr L$ and $\Delta^n$ to $\scr R$ (when the classes $\scr L$ and $\scr R$ are stable under base change along one another, the usual complete Segal anima of spans $\Span(\scr E,\scr L,\scr R)$ is the subobject of $\Lambda(\scr E,\scr L,\scr R)$ consisting of cartesian diagrams).
We denote by $\mathrm N\colon \Cat_\infty\into \Fun(\Delta^\op,\Ani)$ the fully faithful functor given by $\mathrm N(\scr C)=\Lambda(\scr C,\mathrm{iso},\mathrm{all})$, which identifies $\infty$-categories with complete Segal anima.

Let now $p\colon \scr E\to\scr C$ be a cocartesian fibration. If $p^\vee\colon \scr E^\vee\to\scr C^\op$ is the cartesian fibration classifying the same functor $\scr C\to\Cat_\infty$ as $p$, there is by \cite[Theorem 1.4]{cart-dual} a canonical isomorphism
\[
\Lambda(\scr E^\vee,\mathrm{cart},\mathrm{vert})=\Span(\scr E^\vee,\mathrm{cart},\mathrm{vert})\simeq \mathrm N(\scr E)
\]
where ``$\mathrm{cart}$'' and ``$\mathrm{vert}$'' denote the orthogonal classes of cartesian and vertical morphisms, such that the following diagram commute:
\[
\begin{tikzcd}[/tikz/column 2/.append style={anchor=base east}]
	\Lambda(\scr E^\vee,\mathrm{cart},\mathrm{vert}) \arrow[d, hookrightarrow] \ar{r}{\sim} & \mathrm N(\scr E) \arrow["p", start anchor=south, to path={-- (\tikztostart |- \tikztotarget.north) \tikztonodes}]{d} \\
	\Lambda(\scr E^\vee,\mathrm{all},\mathrm{vert}) \ar{r}{p^\vee} & \Lambda(\scr C^\op,\mathrm{all},\mathrm{iso})=\mathrm N(\scr C)\rlap.
\end{tikzcd}
\]

Our strategy is now to define a morphism of simplicial anima
\begin{equation}\label{eqn:Thom-nerve}
\mathrm N(\Vect^{\epi,\op,\otimes}) \to \Lambda((\Sm^{\sncd,\otimes})^\vee,\mathrm{all},\mathrm{vert})
\end{equation}
over $\mathrm N(\Sch^\op\times\Fin_*)$
such that the composite
\[
\mathrm N(\Vect^{\epi,\op,\otimes}) \to \Lambda((\Sm^{\sncd,\otimes})^\vee,\mathrm{all},\mathrm{vert}) \to \Lambda((\scr P_{\ebu}(\Sm)_*^\otimes)^\vee,\mathrm{all},\mathrm{vert})
\]
lands in the subobject $\Lambda((\scr P_{\ebu}(\Sm)_*^\otimes)^\vee,\mathrm{cart},\mathrm{vert})\simeq \mathrm N(\scr P_{\ebu}(\Sm)_*^\otimes)$. Since this isomorphism commutes with the maps to $\mathrm N(\Sch^\op\times\Fin_*)$ and the functor $\mathrm N$ is fully faithful, this yields the desired functor~\eqref{eqn:Thom-cocart}.

The construction of~\eqref{eqn:Thom-nerve} is straightforward using the already established functoriality of the construction $\B$.
To keep the notation reasonable, we only spell out the map~\eqref{eqn:Thom-nerve} on $2$-simplices, but the general case is similar and the simplicial structure will be apparent. A $2$-simplex of $\mathrm N(\Vect^{\epi,\op,\otimes})$ consists of
\[
S\xleftarrow f T\xleftarrow g U,\quad I_+\xrightarrow{\alpha} J_+\xrightarrow{\beta} K_+,\quad
\Bigl({\textstyle\bigoplus_{\alpha(i)=j} f^*(\scr E_i)}\overset{\phi_j}\twoheadleftarrow \scr F_j\Bigr)_{j\in J},\quad
\Bigl({\textstyle\bigoplus_{\beta(j)=k} g^*(\scr F_j)}\overset{\psi_k}\twoheadleftarrow \scr G_k\Bigr)_{k\in K}.
\]
The corresponding $2$-simplex $\Tw(\Delta^2)\to (\Sm^{\sncd,\otimes})^\vee$ is as follows:
\[
\begin{tikzcd}
	(\P(\scr E_i\oplus\scr O))_{i\in I} & (\B(f^*\scr E|_{\alpha^{-1}(j)}))_{j\in J} \ar{l}{b_\Pi} \ar{d}{\phi\circ b_\P} & (\B(g^*f^*\scr E|_{(\beta\circ\alpha)^{-1}(k)}))_{k\in K} \ar{l} \ar{d} \\
	& (\P(\scr F_j\oplus\scr O))_{j\in J} & (\B(g^*\scr F|_{\beta^{-1}(k)}))_{k\in K} \ar{l}{b_\Pi} \ar{d}{\psi\circ b_\P} \\
	& & (\P(\scr G_k\oplus\scr O))_{k\in K}\rlap,
\end{tikzcd}
\]
where the three columns lie in the fibers over $(S,I_+)$, $(T,J_+)$, and $(U,K_+)$, respectively.
The fiber of this diagram over $k\in K$ is the diagram~\eqref{eqn:2-zigzag} for the family of sheaves $(g^*f^*\scr E_i)_{i\in (\beta\circ\alpha)^{-1}(k)}$ and the decomposition $(\beta\circ \alpha)^{-1}(k)\to \beta^{-1}(k)\to\{k\}$.
The fact that the maps~\eqref{eqn:B-functoriality} are $\L_\ebu$-equivalences implies that the horizontal maps become cartesian in $(\scr P_{\ebu}(\Sm)_*^\otimes)^\vee$. This yields the desired morphism~\eqref{eqn:Thom-nerve}, hence the desired functor~\eqref{eqn:Thom-cocart}.

\section{Projective bundle homotopy invariance}

Let $S$ be a derived scheme. We shall write
\[
\MotSp_S=\SpP(\scr P_{\Zar,\ebu}(\Sm_S,\Sp))
\]
and refer to objects of $\MotSp_S$ as \emph{motivic spectra} over $S$. We shall also write
\[
\MotSp_S^\un = \SpP(\scr P_{\Zar,\ebu}(\Sm_S)_*)
\]
for the unstable version of $\MotSp_S$, so that $\MotSp_S=\Sp(\MotSp_S^\un)$. 
There are symmetric monoidal left adjoint functors
\[
\begin{tikzcd}
	\scr P(\Sm_S) \ar{r}{(\ph)_+} & \scr P(\Sm_S)_* \ar{r}{\Sigma^\infty_{\P^1}} \ar[bend right=20]{rr}[swap]{\Sigma^\infty_{\P^1}} & \MotSp_S^\un \ar{r}{\Sigma^\infty} & \MotSp_S\rlap.
\end{tikzcd}
\]
As is customary, we will often omit the functors $\Sigma^\infty_{\P^1}$ and $\Sigma^\infty_{\P^1}(\ph)_+$ from the notation, identifying objects in $\scr P(\Sm_S)_*$ and in $\scr P(\Sm_S)$ with their images in $\MotSp_S^\un$ or in $\MotSp_S$.
	The symmetric monoidal $\infty$-category $\MotSp_S^\un$ was denoted by $\SpP(\mathrm{St}_S^\ex)$ in \cite{AnnalaIwasa2}. We will show below that $\MotSp_S$ is equivalent to the full subcategory of fundamental objects in $\MotSp_S^\un$ (Corollary~\ref{cor:delooping}), which was denoted by $\SpP(\mathrm{St}_S^\ex)^\fd$ in \emph{loc.\ cit.}, but this is not at all obvious from the definitions.
	
	Of course, we do not claim that $\MotSp_S$ is ``the'' $\infty$-category of motivic spectra, which we expect to be a further localization thereof (enforcing in particular Nisnevich descent, and hence smooth blowup excision by Proposition~\ref{prop:Nis-sbu}). Rather, $\MotSp_S$ is the minimal construction to which all the results of this paper apply.
	Note that the full subcategory of either $\MotSp_S$ or $\MotSp_S^\un$ consisting of $\A^1$-invariant Nisnevich sheaves is the Morel–Voevodsky stable $\A^1$-homotopy $\infty$-category over $S$ (since smooth blowup excision holds in the latter \cite[Section 3, Remark 2.30]{MV}). 
	We will occasionally denote by $\L_{\A^1}$ the localization onto this full subcategory.
	
	We note the following facts (and analogous ones for $\MotSp_S^\un$):
	\begin{itemize}
		\item The presheaf of $\infty$-categories $S\mapsto \MotSp_S$ satisfies Zariski descent.
		\item If $S$ is qcqs and $X\in \Sm_S^\fp$, then $\Sigma^\infty_{\P^1}X_+\in \MotSp_S$ is compact. In particular, if $S$ is qcqs, then the $\infty$-category $\MotSp_S$ is compactly generated.
		\item If $S$ is the limit of a cofiltered diagram of derived schemes $S_\alpha$ with affine transition maps, then $\MotSp_S=\lim_\alpha \MotSp_{S_\alpha}$.
	\end{itemize}

\begin{theorem}
	\label{thm:euler}
	Let $\scr E$ be a finite locally free sheaf on $X\in\scr P(\Sm_S)$ and let $\sigma\colon \scr E\to\scr O_X$ be a linear map.
	\begin{enumerate}[beginpenalty=10000]
	\item\label{item:euler} \textnormal{(Euler class of locally free sheaves)}
	There is a canonical homotopy $\bar h(\sigma)$ in $\MotSp_S^\un$ between
	\[
	X_+\xrightarrow\sigma \V(\scr E)_+\subset \P(\scr E\oplus\scr O_X)_+\to \Th_X(\scr E)
	\]
	and the zero section.
	\item\label{item:P-inv} \textnormal{($\P$-homotopy invariance)}
	There is a canonical homotopy $h(\sigma)$ in $(\MotSp_S)_{/X}$ between
	\[
	X\xrightarrow\sigma \V(\scr E)\subset \P(\scr E\oplus\scr O_X)
	\]
	and the zero section.
	\end{enumerate}
	 Moreover, the homotopies $\bar h(\sigma)$ and $h(\sigma)$ are functorial in $(S,X,\scr E,\sigma)$, and they are the identity when $\sigma=0$.
\end{theorem}

\begin{proof}
	We may assume $X=S$, as the general case then follows formally from the functoriality in $(S,\scr E,\sigma)$.
	The matrix
	\[
	e_{21}(\sigma)=\begin{pmatrix}
		\id_{\scr E} & 0 \\
		\sigma & \id_{\scr O}
	\end{pmatrix}\in \Aut(\scr E\oplus \scr O)
	\]
	induces an automorphism $e$ of $\P(\scr E\oplus\scr O)$ sending the zero section to that induced by $\sigma$. Moreover, $e$ fixes $\P(\scr E)$ and hence induces an automorphism $\bar e$ of the Thom space $\Th_S(\scr E)$.
	To prove \ref*{item:euler} (resp.\ \ref*{item:P-inv}), it will therefore suffice to show that $\bar e$ (resp.\ $e$) is homotopic to the identity in $\MotSp_S^\un$ (resp.\ in $(\MotSp_S)_{/S}$).

	To prove that $\bar e$ is homotopic to the identity, we choose a factorization
	\[
	\scr E\xrightarrow\phi\scr F\xrightarrow\tau\scr O
	\] 
	of $\sigma$ (to get a functorial construction we can take for example $\scr F=\scr E$, $\phi=\id$, and $\tau=\sigma$, but it will be useful to distinguish between $\scr E$ and $\scr F$ in the notation).
	We will construct more precisely a homotopy in $\Sp_T(\scr P_{\ebu}(\Sm_S)_*)$, where $T=\Th_S(\scr E\oplus\scr F)$ (this is sufficient by Corollary~\ref{cor:Thom-invertible}).
	Consider the span of $\L_{\ebu}$-equivalences
	\[
	\Th_S(\scr E)\wedge \Th_S(\scr F) \leftarrow B/\partial B \to \Th_S(\scr E\oplus \scr F),
	\]
	where $B=\B(\scr E,\scr F)$ is the blowup of $\P(\scr E\oplus\scr F\oplus\scr O)$ at $\P(\scr E)\sqcup \P(\scr F)$, or equivalently of $\P(\scr E\oplus\scr O)\times \P(\scr F\oplus\scr O)$ at $\P(\scr E)\times\P(\scr F)$.
	The matrix
	\[
	e_{31}(\sigma)=\begin{pmatrix}
		\id_\scr E & 0 & 0 \\
		0 & \id_{\scr F} & 0 \\
		\sigma & 0 & \id_{\scr O}
	\end{pmatrix}\in \Aut(\scr E\oplus\scr F\oplus\scr O)
	\]
	induces an automorphism $e'$ of $\P(\scr E\oplus\scr F\oplus\scr O)$, fixing $\P(\scr E\oplus\scr F)$ and thereby inducing an automorphism $\bar e'$ of $\Th_S(\scr E\oplus \scr F)$.
	 It also induces an automorphism $e''$ of the blowup $B$, since it fixes the center $\P(\scr E)\sqcup \P(\scr F)$, which preserves the boundary $\partial B$ and hence descends to an automorphism $\bar e''$ of the quotient $B/\partial B$. We then have a commutative diagram of pointed presheaves
	\[
	\begin{tikzcd}
		\Th_S(\scr E)\wedge \Th_S(\scr F) \ar{d}[swap]{\bar e\wedge \id} & B/\partial B \ar{l} \ar{r} \ar{d}[swap]{\bar e''} & \Th_S(\scr E\oplus \scr F) \ar{d}{\bar e'} \\
		\Th_S(\scr E)\wedge \Th_S(\scr F) & B/\partial B \ar{l} \ar{r} & \Th_S(\scr E\oplus \scr F)\rlap.
	\end{tikzcd}
	\]
	Since the horizontal maps are $\L_{\ebu}$-equivalences, it suffices to show that $\bar e'$ becomes homotopic to the identity in $\Sp_{T}(\scr P_{\ebu}(\Sm_S)_*)$. 
	In $\Aut(\scr E\oplus\scr F\oplus\scr O)$ we have the commutator relation 
	\begin{equation*}\label{eqn:commutator}
		e_{31}(\sigma)=[e_{32}(\tau),e_{21}(\phi)].
	\end{equation*}
	 Since any lower unitriangular matrix fixes $\P(\scr E\oplus\scr F)$ and hence induces an automorphism of $\Th_S(\scr E\oplus \scr F)$, we deduce that $\bar e'$ is a commutator in the monoid of endomorphisms of $\Th_S(\scr E\oplus \scr F)$.
	However, the object $\Th_S(\scr E\oplus \scr F)$ is invertible in the symmetric monoidal $\infty$-category $\Sp_{T}(\scr P_{\ebu}(\Sm_S)_*)$, so its monoid of endomorphisms has a canonical structure of $\E_\infty$-monoid. The above commutator relation therefore induces a canonical identification of $\bar e'$ with the identity.
	
	We now show that $e$ itself is stably homotopic 
	over $S$ 
	to the identity, in fact that it becomes so after a single suspension in $\Sp_{T}(\scr P_{\ebu}(\Sm_S)_*)$. 
	Since $e$ restricts to the identity on $\P(\scr E)$, it induces an automorphism of cofiber sequences
	\[
	\begin{tikzcd}
		\P(\scr E\oplus\scr O)_+ \ar{r} \ar{d}[swap]{e} & \Th_S(\scr E) \ar{d}[swap]{\bar e} \ar{r}{\delta} & \Sigma(\P(\scr E)_+) \ar{d}{\id} \\
		\P(\scr E\oplus\scr O)_+ \ar{r} & \Th_S(\scr E) \ar{r}{\delta} & \Sigma(\P(\scr E)_+)\rlap.
	\end{tikzcd}
	\]
	Applying the cofiber functor to the endomorphism $(\bar e,\id)$ of $\delta$, we obtain the endomorphism $\Sigma (e_+)$ of $\Sigma(\P(\scr E\oplus\scr O)_+)$.
	We will show that the homotopy between $\bar e$ and the identity in $\Sp_{T}(\scr P_{\ebu}(\Sm_S)_*)$ constructed above can be promoted to a homotopy in the slice category over $\Sigma(\P(\scr E)_+)$. 
	This will in particular give a homotopy between the endomorphism $(\bar e,\id)$ of $\delta$ and the identity in the slice of the arrow category over the arrow $*\to \Sigma(S_+)$.
	Taking the cofiber, we will thus obtain a homotopy over $\Sigma(S_+)$ between $\Sigma (e_+)$ and the identity.

	Let us first explain the categorical aspects of the argument. Consider
	\[
	Y = \Sigma(\P(\scr E)_+)\otimes \Th_S(\scr E)^{-1} \in \Sp_{T}(\scr P_{\ebu}(\Sm_S)_*),
	\]
	so that we may view $\delta$ as a morphism $\delta\colon \1\to Y$.
	The homotopy between $\bar e$ and the identity comes from writing $\bar e$ as a commutator $[a,b]$ of two automorphisms of $\1$.
	We will promote $b$ to an automorphism over $Y$ and show that $\bar e$ over $Y$ can be decomposed as follows:
	\[
	\begin{tikzcd}
		\1 \ar{r}[swap]{a} \ar{d}{\delta} \ar[bend left=10,yshift=4]{rrrr}{\bar e} & \1 \ar{r}[swap]{b} \ar{d}{\delta} & \1 \ar{r}[swap]{a^{-1}} \ar{d}{\delta} & \1 \ar{r}[swap]{b^{-1}} \ar{d}{\delta} & \1 \ar{d}{\delta} \\
		Y \ar{r}{a} \ar[bend right=10,yshift=-4]{rrrr}[swap]{\id} & Y \ar[equal]{r} & Y \ar{r}{a^{-1}} & Y \ar[equal]{r} & Y\rlap,
	\end{tikzcd}
	\]
	where the first and third squares commute via the $\1$-module structure of $\delta$ and the lower cell commutes canonically.
	On the other hand, there is a commutative cube
	\[
	\begin{tikzcd}[row sep={35,between origins}, column sep={35,between origins}]
	      & Y\otimes\1 \ar{rr}{\id\otimes a} \ar[equal]{dd} & & Y\otimes\1 \ar[equal]{dd}  \\
	    \1\otimes\1 \ar[crossing over]{rr}[fill=white]{\id\otimes a} \ar{dd}[swap]{b\otimes\id} \ar{ur}{\delta\otimes\id} & & \1\otimes\1 \ar{ur}{\delta\otimes\id} \\
	      & Y\otimes\1  \ar{rr} & & Y\otimes\1\rlap,  \\
	    \1\otimes\1 \ar{rr}[swap]{\id\otimes a} \ar{ur}[swap]{\delta\otimes\id} && \1\otimes\1 \ar[from=uu,crossing over,"{b\otimes \id}",near start] \ar{ur}[swap]{\delta\otimes\id}
	\end{tikzcd}
	\]
	where the left and right faces are given by $b$ over $Y$ and the morphism between them is multiplication by $a$. 
	This cube provides an identification between the commutator $[a,b]$ and the identity in the slice category over $Y$.
	Thus, it will suffice to decompose $\bar e$ as above.
	
	To that end let
	\[
	B'=\Bl_{\P(\scr E)}\P(\scr E\oplus\scr F\oplus\scr O)\quad\text{and}\quad \partial B'=\Bl_{\P(\scr E)}\P(\scr E\oplus\scr F)\cup E,
	\] 
	where $E=\P(\scr E)\times\P(\scr F\oplus\scr O)$ is the exceptional divisor.
	By Proposition~\ref{prop:blowup-square}, we have $\L_{\ebu}$-equivalences
	\[
	\Th_S(\scr E)\wedge \Th_S(\scr F) \leftarrow B/\partial B\to B'/\partial B'\to \Th_S(\scr E\oplus\scr F).
	\]
	The point is that the homotopy between $\bar e\wedge\id_{\Th(\scr F)}$ and the identity was obtained from a commutator of two automorphisms of $\Th_S(\scr E\oplus\scr F)$, but the morphism of pointed presheaves $\delta\wedge\id_{\Th(\scr F)}$ does not descend to $\Th_S(\scr E\oplus\scr F)$. It does however descend to $B'/\partial B'$, while at the same time the two automorphisms of $\Th_S(\scr E\oplus\scr F)$ lift to $B'/\partial B'$.
	
	Indeed, there is a commutative square of pointed presheaves
	\[
	\begin{tikzcd}[row sep={50,between origins}, column sep={180,between origins}]
	\Th_S(\scr E)\wedge\Th_S(\scr F) \ar{d}[swap]{\delta\wedge\id} & B/\partial B \ar{l} \ar{d}
	\\
	\Sigma(\P(\scr E)_+)\wedge\Th_S(\scr F) & B'/\partial B' \ar{l}[swap]{\delta'} \rlap,
	\end{tikzcd}
	\]
	which identifies $\delta\wedge\id$ with $\delta'$ in $(\scr P_\ebu(\Sm_S)_*)_{/\Sigma(\P(\scr E)_+)\wedge\Th_S(\scr F)}$.
	This square is the cofiber of the following cube in $\scr P(\Sm_S)$:
	\[
	\begin{tikzcd}[row sep={30,between origins}, column sep={100,between origins}]
	      & (\P(\scr E\oplus\scr O)_+\times \P(\scr F))\cup (\P(\scr E)_+\times \P(\scr F\oplus\scr O)) \ar[leftarrow]{rr}\ar{dd}\ar{dl} & & \partial B \ar{dd}\ar{dl} \\
	    \P(\scr E\oplus\scr O)_+\times \P(\scr F\oplus\scr O) \ar[crossing over,leftarrow]{rr} \ar{dd}[swap]{\pi_2} & & B  \\
	      & \P(\scr F)\sqcup_{\P(\scr E)_+\times \P(\scr F)} (\P(\scr E)_+\times \P(\scr F\oplus\scr O))  \ar[leftarrow]{rr} \ar{dl} & &  \partial B' \ar{dl}  \\
	    \P(\scr F\oplus\scr O) \ar[leftarrow]{rr} && B'\rlap. \ar[from=uu,crossing over]
	\end{tikzcd}
	\]
	The bottom face of this cube is functorial with respect to lower unitriangular matrices in $\Aut(\scr E\oplus\scr F\oplus\scr O)$, while the whole cube is functorial with respect to the subgroup $\Hom(\scr E,\scr O)\oplus\Hom(\scr F,\scr O)$. In particular, the matrices $e_{31}(\sigma)$ and $e_{32}(\tau)$ induce automorphisms of the cube, and the matrix $e_{21}(\phi)$ induces an automorphism of the bottom face.
	The given automorphism $\bar e\wedge\id$ of $\delta\wedge\id$ is induced by the matrix $e_{31}(\sigma)$, which acts by the identity on the lower left edge of the cube.
	We now claim that the commutator relation $e_{31}(\sigma)=[e_{32}(\tau),e_{21}(\phi)]$ gives the desired decomposition of $\bar e\wedge\id$.
	Indeed, the matrix $e_{21}(\phi)$ acts by the identity on the lower left edge of the cube.
	Moreover, the automorphism of $\delta\wedge\id$ induced by $e_{32}(\tau)$ is
	\[
	\begin{tikzcd}
		\Th_S(\scr E)\wedge\Th_S(\scr F) \ar{d}[swap]{\id\wedge\bar f} \ar{r}{\delta\wedge\id} & \Sigma(\P(\scr E)_+)\wedge\Th_S(\scr F) \ar{d}{\id\wedge \bar f} \\
		\Th_S(\scr E)\wedge\Th_S(\scr F)\ar{r}{\delta\wedge\id} & \Sigma(\P(\scr E)_+)\wedge\Th_S(\scr F)\rlap,
	\end{tikzcd}
	\]
	where $\bar f$ is given by the matrix $e_{21}(\tau)\in\Aut(\scr F\oplus\scr O)$, so it is multiplication by an automorphism of $\1$ in $\Sp_{T}(\scr P_\ebu(\Sm_S)_*)$, as desired.
\end{proof}

\begin{remark}
	The main results of this paper only use the rank $1$ case of Theorem~\ref{thm:euler}\ref*{item:P-inv}, which is significantly easier to prove. 
	Indeed, when $\scr L$ is an invertible sheaf on $S$, the cofiber sequence
	\[
	S_+=\P(\scr L)_+\to \P(\scr L\oplus\scr O)_+\to \Th_S(\scr L)
	\]
	is split by the structure map $\P(\scr L\oplus\scr O)_+\to S_+$. This yields a canonical decomposition $\P(\scr L\oplus\scr O_S)_+\simeq \1\oplus \Th_S(\scr L)$ in the stable $\infty$-category $\scr P(\Sm_S,\Sp)$, under which $e=\id\oplus \bar e$ (since $e$ commutes with both the inclusion $S=\P(\scr L)\into \P(\scr L\oplus\scr O)$ and its retraction). Thus, $e$ is homotopic to the identity if $\bar e$ is.
\end{remark}

\begin{definition}[$\P^1$-homotopy]
	Let $\scr C$ be an $\infty$-category tensored over $\Sm_\Z^\fp$, and let $f,g\colon X\to Y$ be morphisms in $\scr C$. A \emph{$\P^1$-homotopy} between $f$ and $g$ is a morphism $h\colon \P^1\otimes X\to Y$ making the following diagram commute:
	\[
	\begin{tikzcd}
		*\otimes X \ar{r}{\sim} \ar{d}[swap]{0\otimes\id} & X \ar{d}{f} \\
		\P^1\otimes X \ar[dashed]{r}{h} & Y \\
		*\otimes X \ar{u}{1\otimes\id} \ar{r}{\sim} & X\rlap. \ar{u}[swap]{g}
	\end{tikzcd}
	\]
\end{definition}

\begin{corollary}[$\P^1$-homotopy invariance]
	\label{cor:P^1-homotopy}
	In $\Sp_{\P^1}(\scr P_{\ebu}(\Sm_S,\Sp))$ and hence in $\MotSp_S$, $\P^1$-homotopic morphisms are homotopic.
\end{corollary}

\begin{proof}
	This follows directly from Theorem~\ref{thm:euler}\ref*{item:P-inv}, which implies that the two maps $0,1\colon S\to \A^1\subset\P^1$ are homotopic (noting that the proof does not use Zariski descent in this case).
\end{proof}

\begin{corollary}[Euler class of trivial bundles]
	\label{cor:euler-trivial}
	Let $\scr E$ be a finite locally free sheaf on $X\in\scr P(\Sm_S)$. If there exists an epimorphism $\scr E\onto\scr O$, then the pointed map $X_+\to\Th_X(\scr E)$ induced by the zero section becomes nullhomotopic in $\Sp_{\P^1}(\scr P_{\ebu}(\Sm_S,\Sp))$ and hence in $\MotSp_S$.
\end{corollary}

\begin{proof}
	We may assume $\scr E=\scr O$, since the given map for $\scr E$ factors through the one for $\scr O$. By Corollary~\ref{cor:P^1-homotopy}, the zero section and the section at infinity $S_+\to\P^1_+$ become homotopic in $\Sp_{\P^1}(\scr P_{\ebu}(\Sm_S,\Sp))$, but the latter is nullhomotopic when composed with the quotient map $\P^1_+\to \P^1/\infty= \Th_S(\scr O)$.
\end{proof}

\begin{definition}[Weighted $\A^1$-homotopy]
	Let $\scr C$ be an $\infty$-category tensored over $\Sm_\Z^\fp$, and let $f,g\colon X\to Y$ be morphisms in $\scr C$.
	 A \emph{weighted $\A^1$-homotopy} or \emph{$\A^1/\G_m$-homotopy} from $f$ to $g$ is a $\G_m$-equivariant morphism $h\colon \A^1\otimes X\to Y$, where $\G_m$ acts on $\A^1$ with weight $1$ and trivially on $X$ and $Y$, making the following diagram commute:
	\[
	\begin{tikzcd}
		*\otimes X \ar{r}{\sim} \ar{d}[swap]{0\otimes\id} & X \ar{d}{f} \\
		\A^1\otimes X \ar[dashed]{r}{h} & Y \\
		*\otimes X \ar{u}{1\otimes\id} \ar{r}{\sim} & X\rlap. \ar{u}[swap]{g}
	\end{tikzcd}
	\]
\end{definition}

\begin{remark}\label{rmk:weighted-A^1}
	For $n\in \Z$, let $\A^1(n)$ denote the quotient $\A^1/\G_m$ where $\G_m$ acts with weight $n$.
	We can then define an $\A^1(n)$-homotopy in the obvious way. However, the resulting homotopy relations fall in only two classes:
	\begin{itemize}
		\item If $n=0$, two morphisms are $\A^1(0)$-homotopic if and only if they are $\A^1$-homotopic.
		\item If $n\neq 0$, two morphisms are $\A^1(n)$-homotopic if and only if they are $\A^1(1)$-homotopic (and they are then also $\A^1$-homotopic). Indeed, for any $m\in\Z$, there is a map
	\[
	\A^1(m)\to \A^1(n),\quad t\mapsto t^{\lvert n\rvert},
	\]
	sending $0$ to $0$ and $1$ to $1$.
	\end{itemize}
\end{remark}

\begin{corollary}[Weighted $\A^1$-homotopy invariance]
	\label{cor:A^1/G_m-homotopy}
	In $\MotSp_S$, $\A^1/\G_m$-homotopic morphisms are homotopic.
\end{corollary}

\begin{proof}
	In $\scr P_\Zar(\Sm_S)$, there is a map $\P^1\to\A^1/\G_m$ sending $0$ to $0$ and $1$ to $1$ (classifying the effective Cartier divisor $0\in\P^1$). The claim now follows from Corollary~\ref{cor:P^1-homotopy}.
\end{proof}

\begin{proposition}\label{prop:punctured-P^n}
	Let $\scr E$ and $\scr F$ be finite locally free sheaves on $X\in\scr P(\Sm_S)$. Then the triangle
		\[
		\begin{tikzcd}
			\P(\scr E\oplus\scr F)-\P(\scr F) \ar[hookrightarrow]{dr} \ar[twoheadrightarrow]{d}[swap]{\pi} & \\
			\P(\scr E) \ar[hookrightarrow]{r} & \P(\scr E\oplus\scr F)
		\end{tikzcd}
		\]
	commutes up to homotopy in $(\MotSp_S)_{\P(\scr E)//X}$.
\end{proposition}

\begin{proof}
	We define a $\P$-homotopy
	\[
	h\colon \P_{\P(\scr E\oplus\scr F)-\P(\scr F)}(\scr F(-1)\oplus\scr O) \to \P(\scr E\oplus\scr F)
	\]
	as follows.
	A point in the source is an invertible quotient $\phi\colon \scr E\oplus\scr F\onto\scr L$ such that $\phi|_{\scr E}$ is still surjective and a further invertible quotient $\psi\colon \scr F\otimes\scr L^\vee\oplus\scr O\onto\scr M$. We send this to the quotient
	\[
	\scr E\oplus\scr F\xrightarrow{\phi|_{\scr E}\oplus\id} \scr L\oplus\scr F \xrightarrow{\psi} \scr L\otimes\scr M.
	\]
	If we precompose $h$ with the zero section
	\[
	\P(\scr E\oplus\scr F)-\P(\scr F)\into \V_{\P(\scr E\oplus\scr F)-\P(\scr F)}(\scr F(-1)),
	\]
	we get the lower composite in the given triangle. The diagonal map is obtained via the other canonical section, which sends $\phi\colon \scr E\oplus\scr F\onto\scr L$ to $\scr F\otimes\scr L^\vee\into(\scr E\oplus\scr F)\otimes\scr L^\vee\xrightarrow{\phi}\scr O$.
 Note that this section agrees with the zero section when restricted to $\P(\scr E)$. By Theorem~\ref{thm:euler}\ref*{item:P-inv}, $h$ provides the desired homotopy under $\P(\scr E)$ and over $X$.
\end{proof}

\begin{corollary}\label{cor:linear-embeddings}
	Let $\scr E$ and $\scr F$ be finite locally free sheaves on $X\in \scr P(\Sm_S)$, and let $\sigma,\tau\colon \scr F\to\scr E$ be linear maps. Then the linear embeddings $\P(\scr E)\into\P(\scr E\oplus\scr F)$ induced by $\sigma$ and $\tau$ become homotopic in $(\MotSp_S)_{/X}$.
	In particular, any two linear embeddings $\P^{m}\into \P^n$ become homotopic.
\end{corollary}

\begin{proof}
	This follows from Proposition~\ref{prop:punctured-P^n}, since any linear map $\scr F\to\scr E$ induces a section of $\pi$.
\end{proof}

\begin{corollary}\label{cor:A^1-in-P^1}
	Let $\scr F$ be a finite locally free sheaf on $X\in\scr P(\Sm_S)$. Then the embedding $\V(\scr F)\into \P(\scr F\oplus\scr O)$ becomes homotopic to the constant map $\V(\scr F)\to X\xrightarrow{0} \P(\scr F\oplus\scr O)$ in $(\MotSp_S)_{X//X}$.
\end{corollary}

\begin{proof}
	This is the special case of Proposition~\ref{prop:punctured-P^n} with $\scr E=\scr O$.
\end{proof}

\begin{proposition}[Bass fundamental theorem]
	\label{prop:fundamental}
	The canonical map
	\[
	\partial\colon (\P^1,1) \to \Sigma(\G_m,1)
	\]
	in $\scr P_\Zar(\Sm_S)_*$ admits a retraction in $\MotSp_S$.
\end{proposition}

\begin{proof}
	The map $\partial$ is the cofiber of $(\P^1-\{\infty\},1)\vee(\P^1-\{0\},1)\to (\P^1,1)$. By symmetry, it suffices to show that the inclusion $(\A^1,0)\to(\P^1,0)$ is nullhomotopic. Since $\MotSp_S$ is stable, this inclusion decomposes as $(\A^1,0)\to \A^1_+\to\P^1_+\to (\P^1,0)$, and the map $\A^1_+\to(\P^1,0)$ is nullhomotopic by Corollary~\ref{cor:A^1-in-P^1}.
\end{proof}

Let $\scr V$ be tensored over $\scr P_{\Zar}(\Sm_S)_*$.
Recall that an object $E\in \scr V$ is \emph{fundamental} if the map 
\[\partial\otimes\id_E\colon (\P^1,1)\otimes E\to \Sigma(\G_m,1)\otimes E\]
admits a retraction \cite[Definition 2.3.2]{AnnalaIwasa2}. We write $\scr V^\fd\subset \scr V$ for the full subcategory spanned by the fundamental objects.

\begin{corollary}[Bass delooping]
	\label{cor:delooping}
	Let $\scr V$ be presentably tensored over $\MotSp^\un_S$.
	\begin{enumerate}[beginpenalty=10000]
		\item\label{item:smashing} The adjunction
	\[
	\Sigma^\infty: \scr V \rightleftarrows \Sp(\scr V):\Omega^\infty
	\]
	is a smashing localization, i.e., $\Omega^\infty$ is fully faithful and $\Omega^\infty\Sigma^\infty$ is given by tensoring with $\Omega^\infty\Sigma^\infty\1\in\MotSp^\un_S$. Moreover, the essential image of $\Omega^\infty$ is contained in $\scr V^\fd$.
	\item\label{item:bass} Suppose that $\scr V\otimes_{\MotSp^\un_S}\MotSp^\un_U$ is compactly generated for every qcqs open subscheme $U\subset S$. Then $\Omega^\infty$ induces an isomorphism $\Sp(\scr V)\simeq \scr V^\fd$. In particular, $\scr V^\fd$ is stable and presentable.
	\end{enumerate}
\end{corollary}

\begin{proof}
	The functor $\Omega^\infty$ is a priori lax $\MotSp^\un_S$-linear und commutes with $\SigmaP$, hence preserves fundamental objects.
	By Proposition~\ref{prop:fundamental}, every object of $\Sp(\scr V)$ is fundamental.
	Assertion \ref*{item:bass} now follows from \cite[Theorem 2.4.5]{AnnalaIwasa2} when $S$ is qcqs, and by descent in general.
	In this case, it is clear that the localization is smashing by definition of ``fundamental''.
	In particular, \ref*{item:smashing} holds for $\scr V=\MotSp^\un_S$, hence for arbitrary $\scr V$ by tensoring.
\end{proof}

The following proposition is an adaptation of a result by Panin and Smirnov \cite[Lemma 3.8]{PaninSmirnov}, which is crucial to proving the orientability of $\MGL$ in $\A^1$-homotopy theory:

\begin{proposition}[Euler class of $\scr O(1)$]\label{prop:EulerP1}
	Let $Y=\P_{\P^1}(\scr O(1)\oplus\scr O)$, let $s_0\colon \P^1\into Y$ be the zero section, and let $i\colon \P^1\into Y$ be the inclusion of the fiber at $\infty\in\P^1$. Let $q\colon Y_+\to \Th_{\P^1}(\scr O(1))$ be the quotient map.
	 In $\MotSp_S$, we have the following relation:
	\[
	q\circ s_0\simeq -q\circ i\colon \P^1_+\to \Th_{\P^1}(\scr O(1)).
	\]
\end{proposition}

\begin{proof}
	Let $p\colon Y\to \P^1$ and $r\colon \P^1\to S$ be the structure maps, let $s_\infty\colon \P^1\into Y$ be the section at infinity (whose cofiber is $q$), and let $b\colon Y\to \P^2$ be a map exhibiting $Y$ as the blowup of $\P^2$ in one point, with exceptional divisor $s_0(\P^1)$. Consider the map
	\[
	y=\id-s_0\circ p\colon Y_+\to Y_+.
	\]
	Then $y\circ s_0\simeq 0$. By elementary blowup excision, $y$ descends to a map $\bar y\colon \P^2_+\to Y_+$. By Corollary~\ref{cor:linear-embeddings}, any two linear embeddings of $\P^1$ in $\P^2$ are homotopic, so that $b\circ i\simeq  b\circ s_\infty$. Composing with $\bar y$, we get $y\circ i\simeq  y\circ s_\infty$.
	Now:
	\[
	q\circ s_0\simeq  q\circ (s_0-s_\infty) \simeq  -q\circ y\circ s_\infty \simeq  -q\circ y\circ i\simeq -q\circ (i-s_0\circ p\circ i) \simeq  -q\circ i+q\circ s_0\circ\infty\circ r.
	\]
	To conclude, we show that $q\circ s_0\circ\infty$ is nullhomotopic. 
	Since $0,\infty\colon S_+\to \P^1_+$ are homotopic by Corollary~\ref{cor:P^1-homotopy}, we have $s_0\circ\infty \simeq  i\circ 0 \simeq  i\circ\infty \simeq  s_\infty\circ\infty$. Since $q\circ s_\infty\simeq 0$, the claim is proved.
\end{proof}

\section{Grassmannians and the stack of vector bundles}

Given finite locally free sheaves $\scr E$ and $\scr F$, we denote by $\V(\scr E,\scr F)$ the vector bundle parametrizing linear maps $\scr E\to\scr F$, and by
$\St(\scr E,\scr F)$ the open subscheme parametrizing surjections $\scr E\onto \scr F$. When $\scr F=\scr O^n$, this is the usual Stiefel variety $\St_n(\scr E)$.

\begin{proposition}[Infinite excision]\label{prop:Stiefel}
	Let $\scr E$ and $\scr F$ be finite locally free sheaves on $X\in\scr P(\Sm_S)$ such that there exists an epimorphism $\chi\colon \scr E^m\onto\scr F$ for some $m\geq 0$. Then the open embedding of ind-$X$-schemes
	\[
	\St(\scr E^\infty,\scr F)\into \V(\scr E^{\infty},\scr F)
	\]
	becomes an isomorphism in $\MotSp_S$.
\end{proposition}

\begin{proof}
	We may assume $\rk\scr F\geq 1$, as the assertion is trivial over the summand of $X$ where $\scr F$ has rank $0$.
	We consider the compactification $\V(\scr E^k,\scr F)\subset \P(\scr E^k\otimes\scr F^\vee\oplus\scr O)$.
	Let $Z_k\subset \P(\scr E^k\otimes\scr F^\vee\oplus\scr O)$ be the closure of the complement of $\St(\scr E^k,\scr F)$ in $\V(\scr E^k,\scr F)$ and let $\partial Z_k= Z_k\cap \P(\scr E^k\otimes \scr F^\vee)$.
	We then have for every $k\geq 0$ a Zariski pushout square
	\[
	\begin{tikzcd}
		\St(\scr E^k,\scr F) \ar[hookrightarrow]{r} \ar[hookrightarrow]{d} & \V(\scr E^k,\scr F) \ar[hookrightarrow]{d} \\
		\P(\scr E^k\otimes\scr F^\vee\oplus\scr O)- Z_k \ar[hookrightarrow]{r}{i_k} & \P(\scr E^k\otimes\scr F^\vee\oplus\scr O)-\partial Z_k\rlap.
	\end{tikzcd}
	\]
	By Zariski descent and stability, it therefore suffices to prove that the sequence of open embeddings $i_k$
	induces an isomorphism in the colimit as $k\to\infty$.
	To do so we will construct a diagonal map in the square
	\[
	\begin{tikzcd}
		\P(\scr E^k\otimes\scr F^\vee\oplus\scr O)- Z_k \ar[hookrightarrow]{r}{e_k} \ar[hookrightarrow]{d}[swap]{i_k} & \P(\scr E^{k+m}\otimes\scr F^\vee\oplus\scr O)- Z_{k+m} \ar[hookrightarrow]{d}{i_{k+m}} \\
		\P(\scr E^k\otimes\scr F^\vee\oplus\scr O)-\partial Z_k \ar[hookrightarrow]{r}{e_k} \ar[dashed]{ur}{f_k} & \P(\scr E^{k+m}\otimes\scr F^\vee\oplus\scr O)-\partial Z_{k+m}
	\end{tikzcd}
	\]
	and homotopies making both triangles commute, such that the composite homotopy is the identity.\footnote{If $S$ is qcqs (which does not restrict the generality), we do not actually need this last condition, since $\MotSp_S$ admits a conservative filtered-colimit-preserving functor to a $1$-category. However, our construction does satisfy this condition and shows that the sequence of morphisms $i_k$ is an isomorphism of ind-objects in any context with $\P$-homotopy invariance.}
	Let
	\[
	f_k\colon \P(\scr E^k\otimes\scr F^\vee\oplus\scr O) \to \P(\scr E^{k+m}\otimes\scr F^\vee\oplus\scr O)
	\]
	be the closed immersion induced by the epimorphism
	\[
	 (\scr E^{k+m}\otimes\scr F^\vee)\oplus\scr O
	 \xrightarrow\chi (\scr E^k\otimes\scr F^\vee)\oplus(\scr F\otimes\scr F^\vee)\oplus\scr O
	 \xrightarrow{\mathrm{ev}}
	 (\scr E^k\otimes\scr F^\vee)\oplus \scr O\oplus\scr O
	 \xrightarrow+ (\scr E^{k}\otimes\scr F^\vee)\oplus\scr O.
	\]
	We define a $\P$-homotopy 
	\[
	h_k\colon \P_{\P(\scr E^k\otimes\scr F^\vee\oplus\scr O)}(\scr O(-1)\oplus\scr O) \to \P(\scr E^{k+m}\otimes\scr F^\vee\oplus\scr O)
	\]
	as follows. A point in the left-hand side is a pair of invertible quotients $\phi\colon \scr E^k\otimes\scr F^\vee\oplus\scr O\onto \scr L$ and $\psi\colon \scr L^\vee\oplus\scr O\onto\scr M$. We send this point to the invertible quotient
	\[
	(\scr E^m\otimes\scr F^\vee)\oplus(\scr E^k\otimes\scr F^\vee)\oplus\scr O \xrightarrow\phi (\scr E^m\otimes\scr F^\vee)\oplus\scr L
	\xrightarrow{\ev\circ\chi} \scr O\oplus\scr L\xrightarrow{\psi\otimes\id_{\scr L}} \scr M\otimes\scr L
	\]
	(here we used that $\rk\scr F\geq 1$, so that the evaluation map $\scr F\otimes\scr F^\vee\to\scr O$ is surjective).
	Considering the canonical section and the zero section of the vector bundle $\V(\scr O(-1))$ over $\P(\scr E^k\otimes\scr F^\vee\oplus\scr O)$ and applying Theorem~\ref{thm:euler}\ref*{item:P-inv}, we see that $h_k$ defines a homotopy between $f_k$ and the standard embedding $e_k$. It remains to prove the following three statements:
	\begin{enumerate}
		\item\label{item:Stiefel1} The restriction of $f_k$ to $\P(\scr E^k\otimes\scr F^\vee\oplus\scr O)- \partial Z_k$ lands in $\P(\scr E^{k+m}\otimes\scr F^\vee\oplus\scr O)- Z_{k+m}$.
		\item\label{item:Stiefel2} The restriction of $h_k$ to $\P(\scr E^k\otimes\scr F^\vee\oplus\scr O)- Z_k$ lands in $\P(\scr E^{k+m}\otimes\scr F^\vee\oplus\scr O)- Z_{k+m}$.
		\item\label{item:Stiefel3} The restriction of $h_k$ to $\P(\scr E^k\otimes\scr F^\vee\oplus\scr O)- \partial Z_k$ lands in $\P(\scr E^{k+m}\otimes\scr F^\vee\oplus\scr O)- \partial Z_{k+m}$.
	\end{enumerate}
	The complement of $Z_k$ classifies invertible quotients $\phi\colon \scr E^k\otimes\scr F^\vee\oplus\scr O\onto\scr L$ such that the induced map $\phi^\flat\colon \scr E^k\to\scr F\otimes\scr L$ is surjective, and the complement of $\partial Z_k$ is the union of the latter with $\V(\scr E^k,\scr F)$, which is the locus where $\phi|_{\scr O}$ is surjective (i.e., an isomorphism).
	
	For a pair $(\phi,\psi)$ as above, the map $h_k(\phi,\psi)^\flat$ is the composite
	\[
	\scr E^m\oplus\scr E^k\xrightarrow{\chi} \scr F\oplus\scr E^k\xrightarrow{\phi^\flat} \scr F\otimes(\scr O\oplus\scr L)\xrightarrow{\id_\scr F\otimes\psi\otimes\id_\scr L} \scr F\otimes\scr M\otimes\scr L,
	\]
	and the map $f_k(\phi)^\flat$ is the special case with $\psi\otimes\id_\scr L=\phi|_{\scr O}+\id_\scr L\colon \scr O\oplus\scr L\onto\scr L$.
	Since $\psi$ and $\chi\colon \scr E^m\to\scr F$ are surjective, this composite is surjective if either $\phi^\flat$ is surjective or if $\psi|_{\scr L^\vee}$ is surjective, which proves \ref*{item:Stiefel1} and \ref*{item:Stiefel2}. To prove \ref*{item:Stiefel3}, it remains to show that $h_k$ sends $\V(\scr E^k,\scr F)$ to the complement of $\partial Z_{k+m}$.
	Since the loci where $\psi|_{\scr L^\vee}$ and $\psi|_{\scr O}$ are surjective form an open covering of $\P(\scr O(-1)\oplus\scr O)$, and the case where $\psi|_{\scr L^\vee}$ is surjective is already established, we may assume that $\psi|_{\scr O}$ is surjective. But when both $\phi|_{\scr O}$ and $\psi|_{\scr O}$ are surjective, the map $h_k(\phi,\psi)|_{\scr O}$ is surjective, i.e., $h_k(\phi,\psi)$ belongs to $\V(\scr E^{k+m},\scr F)$.
\end{proof}

\begin{lemma}
	\label{lem:V/G}
	Let $X\in \scr P(\Sch_S)$, let $G$ be a group object in $\scr P(\Sch_S)_{/X}$ containing $\G_m$ as a subgroup, and let $\scr E$ be a finite locally free representation of $G$ over $X$ such that $\G_m$ acts with constant nonzero weight.
	Then the canonical map
	\[
	\V_X(\scr E)/G \to \rm B G
	\]
	is a weighted $\A^1$-homotopy equivalence. In particular, it becomes an isomorphism in $\MotSp_S$.
\end{lemma}

\begin{proof}
	If $\G_m$ acts on $\scr E$ with weight $n$, then the map
	\[
	\A^1/\G_m \times \V_X(\scr E)/G \to \V_X(\scr E)/G,\quad (t,v) \mapsto t^{\lvert n\rvert}v,
	\]
	is an $\A^1/\G_m$-homotopy from the zero map to the identity.
	The last claim follows from Corollary~\ref{cor:A^1/G_m-homotopy}.
\end{proof}

\begin{theorem}[Geometric model of the stack of vector bundles]\label{thm:Gr-Vect}
	Let $\scr E$ be a finite locally free sheaf on $S$ admitting an epimorphism $\scr E\onto\scr O$.
	Then, for every $n\geq 0$, the canonical map
	\[
	\Gr_n(\scr E^\infty)\to \rm B\GL_n=\Vect_n
	\]
	becomes an isomorphism in $\MotSp_S$.
\end{theorem}

\begin{proof}
	This map can be decomposed as
	\[
	\St_n(\scr E^\infty)/\GL_n
	 \to 
	 \V(\scr E^\infty,\scr O^n)/\GL_n
	 \to \rm B\GL_n.
	\]
	The first map is the simplicial colimit of $\St_n(\scr E^\infty)\times \GL_n^\bullet \to \V(\scr E^\infty,\scr O^n)\times \GL_n^\bullet$, hence becomes an isomorphism in $\MotSp_S$ by Proposition~\ref{prop:Stiefel}. The second map is the sequential colimit over $k$ of the maps $\V(\scr E^k,\scr O^n)/\GL_n \to \rm B\GL_n$, hence becomes an isomorphism in $\MotSp_S$ by Lemma~\ref{lem:V/G}.
\end{proof}

\section{Orientations revisited}

Let $E$ be an object in $\MotSp_S^\un$. Recall from \cite[Section 3.1]{AnnalaIwasa2} that $E$ is \textit{orientable} if the map
\[
	[\scr{O}(1)]\otimes\id_E \colon \P^1\otimes E \to \Pic\otimes E
\]
admits a retraction, where $\P^1$ and $\Pic$ are viewed as pointed objects.
If $E\in \MotSp_S^\un$ is orientable, then it is fundamental \cite[Lemma 3.1.7]{AnnalaIwasa2}, hence belongs to the full subcategory $\MotSp_S$ of stable objects (Corollary~\ref{cor:delooping}\ref*{item:bass}).
A choice of such a retraction (in the homotopy category) is called an \textit{orientation} of $E$, and we say that $E$ is \textit{oriented} if an orientation of $E$ is fixed.
An orientation of $E$ defines a cohomology operation $c_1(\scr L)\colon \SigmaP^{-1}E^{X_+}\to E^{X_+}$ for every $X\in\scr P(\Sm_S)$ and every $\scr L\in \Pic(X)$, called the \emph{first Chern class} of $\scr L$.

If $E\in \MotSp_S^\un$ is orientable and has an algebra structure (in the homotopy category), then we can always choose an orientation as a right $E$-linear map. 
When we say that an algebra object $E$ is oriented, we will always assume that the orientation is right $E$-linear.
The first Chern class $c_1(\scr L)\colon \SigmaP^{-1}E^{X_+}\to E^{X_+}$ is then given by left multiplication with a class $c_1(\scr L)\in E^1(X)$.

Suppose that $E\in\MotSp_S^\un$ is oriented.
Then $E$ satisfies projective bundle formula by \cite[Lemma 3.3.5]{AnnalaIwasa2}: for a locally free sheaf $\scr E$ of rank $r$ on $X\in\scr{P}(\Sm_S)$, we have an isomorphism
\[
\sum_{i=0}^{r-1}c_1(\scr O(1))^i\colon \bigoplus_{i=0}^{r-1} \SigmaP^{-i}E^{X_+} \simto E^{\P_X(\scr{E})_+}.
\]
By naturality of $c_1$, we have a commutative square
\[
\begin{tikzcd}
	E^{\P_X(\scr{E}\oplus\scr{O})_+} \ar[r] & E^{\P_X(\scr{E})_+} \\
	\bigoplus_{i=0}^r \SigmaP^{-i}E^{X_+} \ar[u,"\sum c_1(\scr{O}(1))^i","\sim"' sloped] & 
		\bigoplus_{i=0}^{r-1} \SigmaP^{-i}E^{X_+}\rlap, \ar[l] \ar[u,"\sum c_1(\scr{O}(1))^i" swap,"\sim" sloped]
\end{tikzcd}
\]
where the bottom horizontal map is the inclusion of the first $r$ summands. This induces an isomorphism
\[
	t(\scr{E}) \colon \SigmaP^{-r}E^{X_+} \simto E^{\Th_X(\scr{E})}
\]
between the cofiber of the lower map and the fiber of the upper map, called the \emph{Thom isomorphism}.
When $E$ is an algebra, $t(\scr E)$ is right $E^{X_+}$-linear and can be identified with a class $t(\scr E)\in \tilde E^r(\Th_X(\scr E))$, called the \emph{Thom class} of $\scr E$.

The higher Chern classes $c_i(\scr E)\colon \SigmaP^{-i}E^{X_+}\to E^{X_+}$ for $0\leq i\leq r$ are then defined by the formula
\begin{equation}\label{eqn:Thom-Chern}
	t(\scr{E}) = \sum_{i=0}^r (-1)^{r-i}c_1(\scr{O}(1))^i\cdot c_{r-i}(\scr{E}),
\end{equation}
which is to be understood as an equality between maps $\SigmaP^{-r}E^{X_+}\to E^{\P(\scr E\oplus\scr O)_+}$.
In particular, if $s$ a global section of $\scr{E}^\vee$, then the following diagram commutes:
\begin{equation}\label{eqn:Thom}
\begin{tikzcd}
	\SigmaP^{-r}E^{X_+} \ar[r,"t(\scr{E})"] \ar[rd,"c_r(\scr{E})" swap] &
		E^{\Th_X(\scr{E})} \ar[d,"(-1)^rs^*"] \\
	& E^{X_+}\rlap.
\end{tikzcd}
\end{equation}

Using the isomorphism $\Vect_n\simeq\Gr_n$ in $\MotSp_S$ proved in Theorem~\ref{thm:Gr-Vect}, we can considerably simplify the proof of the main result of \cite[Section~4]{AnnalaIwasa2}:

\begin{theorem}[Oriented cohomology of the stack of vector bundles]
	\label{thm:EVect} 
Let $E$ be an oriented object in $\CAlg(\h\MotSp_S)$. Then, for all $X \in \scr{P}(\Sm_S)$ and $n\geq 0$, there is an isomorphism of bigraded rings
\[
	E^{**}(X \times \Vect_n) \simeq E^{**}(X)[[c_1,\dots,c_n]],
\]
where $c_i\in E^{i}(\Vect_n)$ is the $i$th Chern class of the universal rank $n$ locally free sheaf.
\end{theorem}

\begin{proof}
By Theorem \ref{thm:Gr-Vect}, this follows from the computation of the oriented cohomology of Grassmannians as in \cite[Corollary 4.6]{AnnalaIwasa}; see also \cite[Corollary 4.4.5]{AnnalaIwasa2}.
This computation only uses Zariski descent, the projective bundle formula, and the isomorphism $\Pic\simeq\P^\infty$, which is actually a consequence of the first two as proved in Theorem \ref{thm:Gr-Vect}.
\end{proof}

Applying Theorem~\ref{thm:EVect} with $n=1$ and $X=\Pic$ yields in the usual way a formal group law over the graded ring $E^*(S)$. This formal group law computes the first Chern class of the tensor product of two invertible sheaves on any $X\in \scr P(\Sm_S)$ whose image in $\MotSp_S$ is compact (this ensures that Chern classes on $X$ are nilpotent, since $\Pic\simeq\colim_{n}\P^n$).
 
Using weighted $\A^1$-invariance, we can further compute the oriented cohomology of the stack $\rm B \mu_n$ of $\mu_n$-torsors (i.e., the fppf-local delooping of $\mu_n$):

\begin{proposition}[Oriented cohomology of the stack of $\mu_n$-torsors]
	\label{prop:EBmu} 
Let $E \in \MotSp_S$ be oriented and let $n\geq 1$. Then there is a cofiber sequence of motivic spectra
\[
	\SigmaP^{-1} E^{{\rm B \G_m}_+} \xrightarrow{c_1(\scr{L}^{\otimes n})} E^{{\rm B \G_m}_+} \to E^{{\rm B \mu_n}_+},
\]
where $\scr L$ is the universal invertible sheaf on $\rm B \G_m=\Pic$.
Hence, if $E\in\CAlg(\h\MotSp_S)$ is oriented with formal group law $F$, there is for any $X\in \scr P(\Sm_S)$ a long exact sequence
\[
\dotsb \to E^{**}(X)[[c]] \xrightarrow{[n]_F} E^{**}(X)[[c]] \to E^{**}(X\times\rm B \mu_n) \to E^{*+1,*}(X)[[c]]\to\dotsb.
\]
\end{proposition}

\begin{proof}
Let $\P^1(n)$ be the quotient $\P^1 / \G_m$, where $\G_m$ acts with weight $n$ on $\P^1$. In other words, $\P^1(n)$ is the projective bundle $\P_{\rm B \G_m}(\scr{L}^{\otimes n} \oplus \scr O)$. Then $\P^1(n)$ admits an open cover by the weighted lines $\A^1(\pm n)$, such that $\A^1(n) \cap \A^1(-n) \simeq \rm B \mu_n$. Hence, we obtain a cofiber sequence
\[
E^{\P^1(n)_+} \to E^{\A^1(n)_+} \oplus E^{\A^1(-n)_+} \to E^{{\rm B \mu_n}_+}.
\]
Applying the projective bundle formula to $E^{\P^1(n)_+}$, the weighted $\A^1$-invariance to $E^{\A^1(\pm n)_+}$ (Corollary~\ref{cor:A^1/G_m-homotopy}), and the fact that $\scr O (1)$ restricts to $\scr O$ on $\A^1(n)$ and $\scr L^{\otimes n}$ on $\A^1(-n)$, we obtain the desired cofiber sequence.
\end{proof}

Next we prove that an orientation is uniquely recovered from the Thom class of the universal invertible sheaf.
To that end, we introduce the auxiliary notion of \textit{Thom orientation}.
Let $\scr{L}$ be the universal invertible sheaf on $\Pic$ and $\iota$ the canonical map
\[
	\iota \colon \P^1=\Th_*(\scr{L}\vert_*) \to \Th_\Pic(\scr{L}).
\]

\begin{definition}[Thom orientation]\label{def:Thom-orientation}
We say that $E\in\MotSp_S^\un$ is \textit{Thom orientable} if the map
\[
	\iota\otimes\id_E \colon \P^1\otimes E \to \Th_\Pic(\scr{L})\otimes E
\]
admits a retraction.
A choice of such a retraction (in the homotopy category) is called a \textit{Thom orientation} of $E$.
We say that $E$ is \textit{Thom oriented} if a Thom orientation of $E$ is fixed.
\end{definition}

\begin{remark}
If $E\in\MotSp_S^\un$ is oriented, then the map $\Th_\Pic(\scr{L})\otimes E\to \P^1\otimes E$ adjoint to the composite
\[
\SigmaP^{-1}E\to \SigmaP^{-1}E^{\Pic_+}\xrightarrow{t(\scr L)} E^{\Th_{\Pic}(\scr L)}
\]
is a Thom orientation of $E$.
\end{remark}

Let $s_0$ be the zero section of $\V_\Pic(\scr{L})$.
Consider the following diagram in $\scr{P}(\Sm_S)_*$:
\[
\begin{tikzcd}
	S^0 \ar[d] \ar[r] & \Pic_+ \ar[d,"s_0"] \ar[r] & \Pic \ar[ld,dashrightarrow,"\tilde{s}_0"] \\
	\P^1 \ar[r,"\iota"] & \Th_\Pic(\scr{L}) \rlap.
\end{tikzcd}
\]
The left vertical map is canonically nullhomotopic in $\MotSp_S$ by Corollary~\ref{cor:euler-trivial}.
Therefore, we obtain a lift $\tilde{s}_0$ in $\MotSp_S$ as indicated.

\begin{lemma}\label{lem:orientation}
Let $E$ be an object in $\MotSp_S$ with a Thom orientation $t$.
Then the composite
\[
	\Pic\otimes E \xrightarrow{-\tilde{s}_0} \Th_\Pic(\scr{L})\otimes E \xrightarrow{t} \P^1\otimes E
\]
is an orientation of $E$.
\end{lemma}

\begin{proof}
Consider the diagram
\[
\begin{tikzcd}
	\Pic\otimes E \ar[r,"-\tilde{s}_0"] & \Th_\Pic(\scr{L})\otimes E \ar[r,"t"] & \P^1\otimes E \\
	\P^1\otimes E \ar[u] \ar[r,"-\tilde{s}_0"] & \Th_{\P^1}(\scr{O}(1))\otimes E \rlap, \ar[u] &
\end{tikzcd}
\]
where the vertical maps are induced by the map $\P^1\to\Pic$ classifying $\scr{O}(1)$.
The goal is to show that the end-to-end composite is homotopic to the identity.
By the definition of Thom orientation, it suffices to show that the bottom horizontal map $-\tilde{s}_0$ is homotopic to the map induced by $\iota\colon\P^1\to\Th_{\P^1}(\scr{O}(1))$.
This follows from Proposition~\ref{prop:EulerP1}.
\end{proof}

Let $E$ be an object in $\MotSp_S^\un$ and $F$ the internal hom object $\iHom(E,E)$.
Then we define $\Ori(E)$ to be the subset of $\tilde{F}^1(\Pic)$ consisting of orientations of $E$, and we define $\TOri(E)$ to be the subset of $\tilde{F}^1(\Th_\Pic(\scr{L}))$ consisting of Thom orientations of $E$.

\begin{proposition}\label{prop:orientation}
A motivic spectrum $E\in\MotSp_S$ is orientable if and only if it is Thom orientable.
More precisely, there is a bijection
\[
	\Ori(E) \simto \TOri(E)
\]
given by taking the Thom class of the universal invertible sheaf.
Furthermore, if $E$ is an algebra in the homotopy category, then an orientation of $E$ is $E$-linear if and only if the corresponding Thom orientation is $E$-linear.
\end{proposition}

\begin{proof}
We show that the map $\Ori(E)\to\TOri(E)$ given by taking the Thom class of the universal invertible sheaf $\scr{L}$ is a bijection with inverse
\[
	-s_0^* \colon \TOri(E) \to \Ori(E),
\]
which is well-defined by Lemma~\ref{lem:orientation}.
Given an orientation $c=c_1(\scr{L})$, we have
\[
	-s_0^*(t(\scr{L})) = -s_0^*(c_1(\scr{O}(1))-c_1(\scr{L})) = -c_1(\scr O_{\Pic})+ c_1(\scr{L})= c_1(\scr{L}),
\]
where the first equality holds by~\eqref{eqn:Thom-Chern}, the second by the naturality of $c_1$, and the third by definition of an orientation.
It remains to show that $-s_0^*$ is injective, and for this we may assume that $E$ has an orientation $c\in \tilde F^1(\Pic)$.
Then, by the diagram~\eqref{eqn:Thom}, the injectivity of $-s_0^*$ is equivalent to the injectivity of left multiplication by $c$ on $F^*(\Pic)$.
By Theorem~\ref{thm:Gr-Vect}, we have $F^*(\Pic)\simeq F^*(\P^\infty)$.
By the projective bundle formula, $F^*(\P^n)$ is a free right $F^*$-module with basis $1,c,\dotsc,c^n$. It then follows from the Milnor exact sequence that $F^*(\Pic)\simeq\lim_n F^*(\P^n)$. By \cite[Lemma 3.1.8]{AnnalaIwasa2}, we further have $c^{n+1}=0$ in $F^*(\P^n)$, from which we deduce the desired injectivity.
The final statement is obvious.
\end{proof}

\begin{remark}
	In $\MotSp^\un_S$, Thom orientability is a priori a weaker condition than orientability, since there may be Thom orientable objects that are not fundamental.
\end{remark}

\begin{remark}\label{rmk:Thom-variant}
	One can consider a variant of Definition~\ref{def:Thom-orientation} with $\Th_{\P^\infty}(\scr O(1))$ instead of $\Th_{\Pic}(\scr L)$. Since $\P^\infty\simeq\Pic$, the proofs of Lemma~\ref{lem:orientation} and Proposition~\ref{prop:orientation} go through for this definition and imply that it is in fact equivalent to Definition~\ref{def:Thom-orientation} for objects of $\MotSp_S$. However, we do not know if $\Th_{\P^\infty}(\scr O(1))$ and $\Th_{\Pic}(\scr L)$ are actually isomorphic in $\MotSp_S$.
\end{remark}

\section{Algebraic cobordism and the universal orientation}

We consider the symmetric monoidal natural transformation
\[
\Th\colon \Vect\to \MotSp^\un\colon \Sch^\op\to \CAlg(\Cat_\infty)
\]
constructed in Section~\ref{sec:thom}.
By Corollary~\ref{cor:Thom-invertible}, it lands in the presheaf of $\E_\infty$-groups $\Pic(\MotSp^\un)$, which is a Zariski sheaf. Hence, it factors through the Zariski-local group completion of $\Vect$, which coincides with the Zariski sheafification of connective algebraic K-theory. We therefore obtain a symmetric monoidal natural transformation
\[
\Th\colon \K \to \MotSp^\un\colon \Sch^\op\to \CAlg(\Cat_\infty).
\]
Using the general formalism of Thom spectra/relative colimits developed in \cite[Section 16]{norms}, we obtain a symmetric monoidal functor
\[
\M\colon \scr P(\Sm_S)_{/\K} \to \MotSp_S^\un,
\]
natural in $S$. As in \emph{loc.\ cit.}, we then define $\MGL=\M(e)$, where $e\colon \K_{\rk=0}\into \K$ is the kernel of the rank map $\rk\colon \K\to\underline{\Z}$. Explicitly, we have
\[
\MGL=\colim_{(X,\xi)\in\scr K_{\rk=0}}\Th_X(\xi),
\]
where $\scr K_{\rk=0}\to\Sm_S$ is the cartesian fibration classified by $\K_{\rk=0}$.
Since $e$ is an $\E_\infty$-map, $\MGL$ is an $\E_\infty$-algebra. Moreover, $\MGL$ is stable under arbitrary base change $T\to S$, since K-theory is Zariski-locally left Kan extended from smooth schemes \cite[Example A.0.6]{EHKSY3}.
The $\A^1$-localization of $\MGL$ is exactly Voevodsky's algebraic cobordism spectrum (using the description of the latter from \cite[Theorem 16.13]{norms}).

The periodic version is similarly defined by
\[
	\PMGL = \M(\id_\K) = \colim_{(X,\xi)\in\scr{K}}\Th_X(\xi),
\]
where $\scr{K}\to\Sm_S$ is the cartesian fibration classified by $\K$.
Then $\PMGL$ is an $\E_\infty$-algebra and is stable under arbitrary base change.

We will denote by $\MGL(n)$ the Thom spectrum of the map $\Vect_n\to \K$, $\scr E\mapsto \scr E-\scr O^{n}$, that is:
\[
\MGL(n) = \Sigma_{\P^1}^{-n}\Th_{\Vect_n}(\mathcal E_n),
\]
where $\scr E_n\in \Vect_n(\Vect_n)$ is the universal locally free sheaf of rank $n$.

\begin{proposition}\label{prop:MGL-colim}
	The canonical map $\Vect_\infty=\colim_n\Vect_n\to \K_{\rk=0}$ induces an isomorphism in $\MotSp_S$
	\[
	\colim_{n} \MGL(n)\simeq \MGL.
	\]
\end{proposition}

\begin{proof}
	The canonical map $f\colon \Vect_\infty\to \K_{\rk=0}$ is acyclic in the $\infty$-topos of Zariski sheaves on $\Sm_S$ \cite[Lemma 2.1.1]{EHKSY3}, which means that its pushout along itself is an isomorphism.
	Since the Thom spectrum functor $\M$ induces a colimit-preserving functor
	 \[
	 \M\colon \scr P_{\Zar}(\Sm_S)_{/\K} \to \MotSp_S^\un,
	 \]
	 cf.\ \cite[Proposition 16.9(1)]{norms},
	 we obtain a pushout square
	 \[
	 \begin{tikzcd}
	 	\M(e\circ f)\ar{r} \ar{d} & \M(e) \ar{d} \\
		\M(e) \ar{r} & \M(e)
	 \end{tikzcd}
	 \]
	 in $\MotSp_S^\un$.
	 In the stabilization $\MotSp_S$, this square becomes a pullback square, which proves the claim.
\end{proof}

\begin{remark}
	Similarly, we have an isomorphism $\PMGL\simeq\colim_n\SigmaP^{-n}\Th_{\Vect}(\scr E)$ in $\MotSp_S$, where $\scr E\in\Vect(\Vect)$ is the universal finite locally free sheaf.
\end{remark}

The canonical map $\MGL(1)\to\MGL$ is clearly a Thom orientation of $\MGL$, which in turn gives a canonical orientation of $\MGL$ in $\MotSp_S$ by Proposition~\ref{prop:orientation}.
We now prove the universality of $\MGL$ as an oriented ring spectrum.

\begin{lemma}[Multiplicativity of Thom classes]
	\label{lem:Whitney}
Let $E$ be an oriented object in $\CAlg(\h\MotSp_S)$.
Let $\scr{E},\scr{F}$ be finite locally free sheaves on $X,Y\in\scr{P}(\Sm_S)$, respectively.
Then the Thom class $t(\scr{E}\boxplus\scr{F})$ is identified with the external product of the Thom classes $t(\scr{E})\times t(\scr{F})$ under the canonical isomorphism
\[
	\tilde E^*(\Th_{X\times Y}(\scr{E}\boxplus\scr{F})) \simeq \tilde E^*(\Th_X(\scr{E})\wedge\Th_Y(\scr{F})).
\]
\end{lemma}

\begin{proof}
We may assume that $\scr{E}$ and $\scr{F}$ are the universal sheaves on $\Vect_m$ and $\Vect_n$, respectively.
Let $s$ be the zero section of $\V(\scr{E}\boxplus\scr{F})\simeq\V(\scr{E})\times\V(\scr{F})$.
Consider the diagram
\[
\begin{tikzcd}[row sep=tiny]
	& \tilde E^*(\Th_{\Vect_m\times\Vect_n}(\scr{E}\boxplus\scr{F})) \ar[rd,"s^*"] \ar[dd,"\sim" sloped] & \\
	E^*(\Vect_m\times\Vect_n) \ar[ru,"t(\scr{E}\boxplus\scr{F})","\sim"' sloped] \ar[rd,"t(\scr{E})\times t(\scr{F})" swap,"\sim" sloped] & & E^*(\Vect_m\times\Vect_n) \\
	& \tilde E^*(\Th_{\Vect_m}(\scr{E})\wedge\Th_{\Vect_n}(\scr{F})). \ar[ru,"s^*" swap] &
\end{tikzcd}
\]
The right triangle commutes, since the maps
\[
\P(\scr E\oplus\scr O)\times \P(\scr F\oplus\scr O)\xleftarrow{b_\Pi}\B(\pi_1^*\scr E,\pi_2^*\scr F)\xrightarrow{b_{\P}}\P((\scr E\boxplus\scr F)\oplus \scr O)
\]
inducing the vertical isomorphism are both isomorphisms over the open $\V(\scr E)\times\V(\scr F)$.
The boundary of the diagram commutes since
\[
	(-1)^{m+n}s^*(t(\scr{E}\boxplus\scr{F})) = c_{m+n}(\scr{E}\boxplus\scr{F}) = c_m(\scr{E})\times c_n(\scr{F}) = (-1)^{m+n}s^*(t(\scr{E})\times t(\scr{F})),
\]
where the first and third equalities hold by \eqref{eqn:Thom} and the second by the Whitney sum formula \cite[Lemma 4.4.3]{AnnalaIwasa2}.
Furthermore, the map $s^*$ is injective since
\[
	E^*(\Vect_m\times\Vect_n) \simeq E^*(S)[[c_1(\pi_1^*\scr{E}),\dotsc,c_m(\pi_1^*\scr{E}),c_1(\pi_2^*\scr{F}),\dotsc,c_n(\pi_2^*\scr{F})]]
\]
by Theorem~\ref{thm:EVect}.
Therefore, the left triangle commutes as desired.
\end{proof}

\begin{proposition}\label{prop:MGL}
Let $E$ be an oriented object in $\CAlg(\h\MotSp_S)$.
Then there is a unique isomorphism
\[
	E^{**}(\Vect_\infty)\simeq E^{**}(\MGL) 
\]
lifting the Thom isomorphisms $E^{**}(\Vect_n)\simeq E^{**}(\MGL(n))$ for all $n\geq 0$.
\end{proposition}

\begin{proof}
	By Proposition~\ref{prop:MGL-colim}, we have $\MGL=\colim_n\MGL(n)$.
We apply Lemma~\ref{lem:Whitney} to the pair of the universal locally free sheaf $\scr{E}_n$ on $\Vect_n$ and the sheaf $\scr{O}$ on $S$.
Then it follows that the diagram 
\[
\begin{tikzcd}[column sep=35]
	E^{**}(\Vect_{n+1}) \ar[r,"\sim","t(\scr{E}_{n+1})" swap] \ar[d,twoheadrightarrow] & E^{**}(\MGL(n+1)) \ar[d,twoheadrightarrow] \\
	E^{**}(\Vect_n) \ar[r,"\sim","t(\scr{E}_n)" swap] & E^{**}(\MGL(n))
\end{tikzcd}
\]
commutes, where the left vertical map is induced by the map $\Vect_n\to\Vect_{n+1}$ classifying $\scr{E}_n\oplus\scr{O}$, and is surjective by Theorem~\ref{thm:EVect}.
By taking limits and using the Milnor exact sequence, we obtain the desired isomorphism.
\end{proof}

\begin{theorem}[Universality of $\MGL$]
	\label{thm:universality}
$\MGL$ is the initial oriented object in $\CAlg(\h\MotSp_S)$, i.e., for every oriented object $E$ in $\CAlg(\h\MotSp_S)$, there is a unique orientation-preserving morphism $\MGL\to E$ in $\CAlg(\h\MotSp_S)$.
\end{theorem}

\begin{proof}
Let $E$ be an oriented object in $\CAlg(\h\MotSp_S)$.
Let $t\colon\MGL\to E$ be the morphism in $\MotSp_S$ corresponding to $1\in E^0(\Vect_\infty)\simeq E^0(\MGL)$, where the isomorphism is that of Proposition \ref{prop:MGL}.
Then $t$ obviously preserves orientations.
Let $t_n$ denote the restriction of $t$ to $\MGL(n)$.
Then the diagram
\[
\begin{tikzcd}[column sep=30]
	\MGL(n)\otimes\MGL(m) \ar[r,"t_n\otimes t_m"] \ar[d] & E\otimes E \ar[d] \\
	\MGL(n+m) \ar[r,"t_{n+m}"] & E
\end{tikzcd}
\]
commutes by Lemma~\ref{lem:Whitney}.
Also, $t$ preserves units by construction.
Hence, $t$ is a morphism in $\CAlg(\h\MotSp_S)$.

It remains to show the uniqueness.
Suppose that we are given another morphism $t'\colon\MGL\to E$ in $\CAlg(\h\MotSp_S)$, which preserves orientations.
Since $E^0(\MGL)=\lim_n E^0(\MGL(n))$, it suffices to show that the restriction of $t'$ to $\MGL(n)$, which we denote by $t'_n$, agrees with $t_n$ for each $n\ge 1$.
This is clear for $n=1$, because both $t_1$ and $t'_1$ are given by the Thom class of the universal invertible sheaf.
Since $t$ and $t'$ are morphisms of commutative algebras, $t_n$ and $t'_n$ agree with each other when restricted to $\MGL(1)^{\otimes n}$.
However, it follows from Theorem~\ref{thm:EVect} that the map
\[
	E^0(\MGL(n)) \to E^0(\MGL(1)^{\otimes n})
\]
is injective, and thus $t_n=t'_n$.
This completes the proof.
\end{proof}

\begin{remark}
	Slightly more generally, the above argument shows that $\MGL$ is initial among oriented associative algebras in $\h\MotSp_S$ whose orientation class is central (equivalently, whose Thom isomorphisms are bimodule maps). However, $\MGL$ is not initial among oriented associative algebras: this would imply that orientations are equivalent to $\MGL$-module structures, but the free oriented motivic spectrum on the unit, $\colim_n \MGL(1)^{\otimes n}$, is not an $\MGL$-module.
	In fact, an $\MGL$-module is precisely an oriented spectrum whose Thom isomorphisms are compatible with direct sums.
\end{remark}

\begin{remark}
	By Theorem~\ref{thm:Gr-Vect}, the map $\Gr_\infty\to\Vect_\infty$ becomes an isomorphism in $\MotSp_S$.
	But this does not imply that it induces an isomorphism of Thom spectra, i.e., that $\MGL$ is a colimit of Thom spectra over Grassmannians as in $\A^1$-homotopy theory.
	We suspect that this is nevertheless the case. Denoting by $\M\Gr$ the latter colimit, we note that the map $\phi\colon \M\Gr\to\MGL$ is an isomorphism from the perspective of any oriented object in $\CAlg(\h\MotSp_S)$. If we could promote $\phi$ to a morphism of commutative algebras in $\h\MotSp_S$, then $\M\Gr$ would be oriented by Remark~\ref{rmk:Thom-variant} and we would deduce that $\phi$ is an isomorphism. It seems possible to construct such a monoid structure by imitating \cite[Section~2.1]{Panin:2008} and using the results of Section~\ref{sec:thom}.
\end{remark}

We say that $E\in\CAlg(\h\MotSp_S)$ is \textit{periodic} if a unit $\beta\in E^{-1}(\1)$, called the \textit{Bott element}, is given.
Note that $\PMGL$ is periodic with the Bott element given by the canonical map $\P^1=\Th_*(\scr O) \to\Th_{\Pic}(\scr{E}_1)$.

\begin{corollary}[Universality of $\PMGL$]
	\label{cor:universality}
$\PMGL$ is the initial periodic oriented object in $\CAlg(\h\MotSp_S)$, i.e., for every periodic oriented object $E$ in $\CAlg(\h\MotSp_S)$, there is a unique morphism $\PMGL\to E$ in $\CAlg(\h\MotSp_S)$ that preserves the orientation and the Bott element.
\end{corollary}

\begin{proof}
This follows immediately from Theorem~\ref{thm:universality}.
\end{proof}

For later purposes, we record the computation of the oriented homology of $\MGL$. This is a standard computation once we know that $\MGL\otimes (\Gr_{n,k})_+$ is a finite free $\MGL$-module (see for example \cite[Proposition 6.2]{Naumann:2009} for the analogue in $\A^1$-homotopy theory).

\begin{proposition}\label{prop:homology_MGL}
Let $E$ be an oriented object in $\CAlg(\h\MotSp_S)$.
\begin{enumerate}[beginpenalty=10000]
	\item\label{item:hom-vect} There is an isomorphism of $E_{**}$-algebras
	\[
	E_{**}(\Vect_\infty) \simeq E_{**}[\beta_0,\beta_1,\dotsc]/(\beta_0-1),
	\]
	where the ring structure on the left-hand side comes from the algebra structure of $\SigmaP^\infty(\Vect_{\infty})_+\simeq\SigmaP^\infty(\K_{\rk=0})_+$ and $\beta_i \in E_i(\Pic)$ is the predual basis to $c^i\in E^i(\Pic)$.
	\item\label{item:hom-MGL} There is an isomorphism of $E_{**}$-algebras
\[
	E_{**}(\MGL) \simeq E_{**}[b_0,b_1,\dotsc]/(b_0-1),
\]
where $b_i$ is the image of $\beta_i$ under the Thom isomorphism $E_{i}(\MGL)\simeq E_{i}(\Vect_\infty)$. 
Moreover, if $c_E$ and $c_\MGL$ are the images in $(E\otimes\MGL)^1(\Pic)$ of the orientations of $E$ and $\MGL$, we have
\[
c_\MGL = \sum_{i\geq 0}b_i c_E^{i+1}.
\]
\end{enumerate}
\end{proposition}

\begin{proof}
	\ref*{item:hom-vect} By \cite[Lemma 4.4.4]{AnnalaIwasa2}, the Grassmannian formula holds for all $\MGL$-modules $M$ in $\MotSp_S$:
	the map
	\[
		\sum_\alpha c(\scr{Q})^\alpha \colon \bigoplus_\alpha \SigmaP^{-\lVert\alpha\rVert}M \to M^{(\Gr_{n,k})_+}
	\]
	is an isomorphism, where $\alpha=(\alpha_1,\dotsc,\alpha_n)$ runs over all $n$-tuples of non-negative integers with $\sum_i\alpha_i\le k-n$ and we write $\lVert\alpha\rVert=\sum_i i\alpha_i$ and $c^\alpha=\prod_i c_i^{\alpha_i}$.
	It follows that $\MGL\otimes(\Gr_{n,k})_+$ is a finite free $\MGL$-module.
	Hence, for a commutative $\MGL$-algebra $E$ in $\h\MotSp_S$, the map
	\[
		E^{**}(\Gr_{n,k}) \to E_{**}(\Gr_{n,k})^\vee
	\]
	is an isomorphism of $E_{**}$-modules for finite $k$ and thus for $k=\infty$ too.
	Then it follows from Theorem~\ref{thm:Gr-Vect} that $E^{**}(\Vect_n)$ is the dual of $E_{**}(\Vect_n)$.
	Since $E^{**}(\Vect_n)=(E^{**}(\Pic)^{\widehat\otimes n})^{\Sigma_n}$, we have $E_{**}(\Vect_n)=\Sym^nE_{**}(\Pic)$.
	Moreover, the direct sum pairing $\Vect_m\times\Vect_n\to\Vect_{m+n}$ induces the canonical map $\Sym^m\otimes \Sym^n\to\Sym^{m+n}$ in homology.
	The map $E_{**}\to E_{**}(\Pic)$ induced by the base point of $\Pic$ is multiplication by $\beta_0$, and hence so is the map $E_{**}(\Vect_n)\to E_{**}(\Vect_{n+1})$ induced by $\scr E\mapsto\scr E\oplus\scr O$.
	Thus, under the identification
	\[
		E_{**}(\Vect_n)=\Sym^n\biggl(\bigoplus_{i\ge 0}E_{**}\beta_i\biggr) \simeq \Sym^{\le n}\biggl(\bigoplus_{i\ge 1}E_{**}\beta_i\biggr)
	\]
	given by $\beta_0\mapsto 1$, the map $E_{**}(\Vect_n)\to E_{**}(\Vect_{n+1})$ corresponds to the inclusion $\Sym^{\leq n}\to \Sym^{\leq n+1}$.
	In the colimit, we obtain the claimed isomorphism of $E_{**}$-algebras
	\[
		E_{**}(\Vect_\infty) = \colim_nE_{**}(\Vect_n) \simeq E_{**}[\beta_1,\beta_2,\dotsc].
	\]
	
	\ref*{item:hom-MGL} Since we have Thom isomorphisms $t(\scr E_n)\colon M^{\Vect_{n,+}}\simto M^{\MGL(n)}$ for all $\MGL$-modules $M$ in $\MotSp_S$, we get an isomorphism of $\MGL$-modules
	\[
	\MGL\otimes\MGL(n)\simeq \MGL\otimes \Vect_{n,+},
	\]
	hence an isomorphism of $E_{**}$-modules $\tilde E_{**}(\MGL(n))\simeq E_{**}(\Vect_{n})$.
	It follows from Lemma~\ref{lem:Whitney} that this isomorphism is natural in $n$, and that we obtain an isomorphism of rings $E_{**}(\MGL)\simeq E_{**}(\Vect_\infty)$ in the colimit. By definition of $\beta_i$, the last formula is equivalent to the following statement: the map $\tilde E_*(\Pic)\to E_{*-1}(\MGL)$ induced by the universal orientation $c\colon \Pic\to\SigmaP\MGL$ sends $\beta_{i+1}$ to $b_i$. By definition, $c$ factors through $-\tilde s_0\colon \Pic\to \SigmaP\MGL(1)$, and so we must show that the induced map $\tilde E_*(\Pic)\to \tilde E_{*-1}(\MGL(1))$ composed with the Thom isomorphism $\tilde E_{*-1}(\MGL(1))\simeq E_{*-1}(\Pic)$ sends $\beta_{i+1}$ to $\beta_i$. Dualizing, this is equivalent to $-\tilde s_0^*\circ t(\scr E_1)\colon E^{*-1}(\Pic)\to \tilde E^*(\Pic)$ being multiplication by $c$, which is a special case of~\eqref{eqn:Thom}.
\end{proof}

\begin{corollary}\label{cor:homology_MGL}
	Let $E$ be an oriented object in $\CAlg(\h\MotSp_S)$. Then there is an isomorphism of $E$-algebras
	\[
	E\otimes\MGL \simeq E[b_1,b_2,\dotsc]=\bigoplus_{m}\SigmaP^{\deg(m)}E,
	\]
	where $m$ ranges over the monomials in the variables $b_i$ and $\deg(b_i)=i$.
\end{corollary}

\begin{proof}
	Proposition~\ref{prop:homology_MGL}\ref*{item:hom-MGL} gives a map of $E$-algebras from the right-hand side to the left-hand side. It is an isomorphism since Proposition~\ref{prop:homology_MGL} holds not just over $S$ but also over any smooth $S$-scheme.
\end{proof}

\section{Algebraic Conner--Floyd isomorphism}
\label{sec:CF}

We shall prove the Conner--Floyd isomorphism for algebraic K-theory by following the argument of Spitzweck and Østvær in the $\A^1$-invariant setting \cite{Spitzweck-Bott}, i.e., by comparing universal properties of cohomology theories defined on \emph{compact} motivic spectra. 
A key input is the isomorphism $\Sigma^\infty_{\P^1}(\Gr_n)_+\simeq \Sigma^\infty_{\P^1}(\Vect_n)_+$ of Theorem~\ref{thm:Gr-Vect}.
We first introduce some terminology for such cohomology theories:

\begin{definition}
	Let $S$ be a qcqs derived scheme.
	\begin{itemize}
		\item A \emph{cohomology theory} on $\MotSp_S^\omega$ is a homological functor 
	\[E^0\colon \MotSp_S^{\omega,\op}\to \Ab,\]
	i.e, a functor that preserves finite products and sends cofiber sequences to exact sequences.
	We then write $E^q=E^0\circ\SigmaP^{-q}$, $E^{p,q}=E^q\circ\Sigma^{2q-p}$, and we denote by \[\widehat E{}^{p,q}\colon \MotSp_S^\op\to\Pro(\Ab)\] the extension of $E^{p,q}$ that preserves cofiltered limits (which is again a homological functor).
	For a presheaf $X\in\scr P(\Sm_S)$ we write $\widehat E^{p,q}{}(X)$ instead of $\widehat E{}^{p,q}(\Sigma^\infty_{\P^1}X_+)$.
	\item A \emph{ring cohomology theory} will mean a commutative monoid in cohomology theories, with respect to the Day convolution in $\Fun(\MotSp_S^{\omega,\op},\Ab)$.
	\item A \emph{periodic cohomology theory} is a ring cohomology theory $E^0$ with a unit $\beta\in E^{-1}(\1)$.
	\item An \emph{oriented cohomology theory} is a ring cohomology theory $E^0$ with an element $c\in \widehat E^1(\Pic)= \lim_n E^1(\P^n)$, whose restriction to $E^1(\P^1)\simeq E^1(\1)\oplus E^0(\1)$ is $(0,1)$.
	\item A \emph{$\G_m$-preoriented cohomology theory} is a ring cohomology theory $E^0$ with an element $u\in \widehat E^0(\Pic)$ such that $u|_{\mathrm{pt}}=1$ and $\mathord\otimes^*(u)=u_1u_2$, where $u_i=\pi_i^*(u)\in\widehat E^0(\Pic\times\Pic)$.
	We then define the \emph{Bott element} $\beta\in E^{-1}(\1)= \tilde E^0(\P^1)$ to be the element $1-u|_{\P^1}$, and we say that $(E^0,u)$ is \emph{$\G_m$-oriented} if $\beta$ is unit.\footnote{The element $u$ is automatically a unit since $\Pic$ is a group. It should be understood as defining a preorientation of the group scheme $\G_m$ over the ring $E$ in the sense of Lurie \cite[Definition 3.2]{LurieElliptic}.}
	\end{itemize}
\end{definition}

\begin{remark}\label{rmk:E-hat}
	Let $X\in\MotSp_S$ and let $E^0$ be a cohomology theory on $\MotSp_S^\omega$. Then there is a canonical iso\-morphism between the limit of the pro-group $\widehat E^{p,q}(X)$ and the group of natural transformations $X^0(\ph)\to E^{p,q}(\ph)$ on $\MotSp_S^\omega$.
\end{remark}

Let $(E^0,c)$ be an oriented cohomology theory on $\MotSp_S^\omega$. For an arbitrary presheaf $X\in\scr P(\Sm_S)$ and an invertible sheaf $\scr L\in \Pic(X)$, we define the first Chern class $c_1(\scr L)\in\widehat E^1(X)$ to be the pullback of $c$ along the map $X\to\Pic$ classifying $\scr L$.
The theory $\widehat E^{**}$ then satisfies the projective bundle formula: for any $X\in\scr P(\Sm_S)$ and any locally free sheaf $\scr E$ of rank $n$ over $X$, the map of pro-groups
\[
\bigoplus_{i=0}^{n-1}\widehat E^{*-2i,*-i}(X)\to \widehat E^{**}(\P(\scr E)),\quad (a_0,\dotsc,a_{n-1})\mapsto \sum_{i=0}^{n-1}c_1(\scr O_{\P(\scr E)}(1))^i p^*(a_i),
\]
is an isomorphism. To see this, consider the full subcategory of $\scr P(\Sm_S)_{/X}$ where the projective bundle formula holds for the pullback of $\scr E$.
This subcategory contains representable presheaves by the proof of \cite[Lemma 3.3.5]{AnnalaIwasa2}. It also contains the initial object, is closed under pushouts by the 5-lemma, and is closed under filtered colimits by definition of $\widehat E^{**}$. It therefore contains $X$ itself. Consequently, we also have the Thom isomorphism
\[
\widehat E^{**}(\Th_X(\scr E))\simeq \widehat E^{*-2n,*-n}(X)
\]
and higher Chern classes $c_i(\scr E)\in \widehat E^i(X)$.

One can further compute the ring structure on the cohomology of a finite product of projective spaces as in \cite[Lemma 3.1.8]{AnnalaIwasa2}, using that for a scheme $X\in \Sm_S^\fp$ and quasi-compact open subschemes $U_1,\dotsc,U_n\subset X$ we have a refined cup product
\[
E^{**}(X/U_1)\otimes\dotsb\otimes E^{**}(X/U_n)\to E^{**}(X/(U_1\cup\dotsb\cup U_n)).
\]
Together with the isomorphism $\Sigma^\infty_{\P^1}\P^\infty_+\simeq \Sigma^\infty_{\P^1}\Pic_+$ of Theorem~\ref{thm:Gr-Vect}, we obtain an isomorphism of pro-rings
\[
\widehat E^{**}(\Pic^{n}) = E^{**}[[x_1,\dotsc,x_n]],
\]
where $x_i=\pi_i^*(c)$. Since $(\Pic,\otimes)$ is an $\E_\infty$-group, the power series $\mathord\otimes^*(c)\in E^{*}[[x_1,x_2]]$ is a commutative formal group law over $E^{*}$, homogeneous of cohomological degree $1$. By construction, this formal group law computes $c_1(\scr L_1\otimes\scr L_2)$ in terms of $c_1(\scr L_1)$ and $c_1(\scr L_2)$ for any $X\in\scr P(\Sm_S)$ and any $\scr L_1,\scr L_2\in \Pic(X)$ (and first Chern classes are nilpotent when $\Sigma^\infty_{\P^1}X_+$ is compact). Using the formal group law, one may prove the Whitney sum formula for the Chern classes in $\widehat E^*$ exactly as in \cite[Lemma 4.4.3]{AnnalaIwasa2}.

The following lemma explains the relationship between orientations and $\G_m$-orientations:

\begin{lemma}[Orientations vs.\ $\G_m$-orientations]
	\label{lem:G_m-orientation}
	Let $E^0$ be a ring cohomology theory on $\MotSp_S^\omega$. Then the assignment
	\[
	u\mapsto(\beta,c),\quad \beta=1-u|_{\P^1},\quad c =\beta^{-1}(1-u),
	\]
	gives a bijection between $\G_m$-orientations of $E^0$ and pairs $(\beta,c)$ consisting of a unit $\beta\in E^{-1}(\1)$ and an orientation $c\in\widehat E^1(\Pic)$ satisfying
		\[
		\mathord\otimes^*(c)=x_1+x_2-\beta x_1x_2,
		\]
		where $x_i=\pi_i^*(c)\in\widehat E^1(\Pic\times\Pic)$. The inverse is given by $(\beta,c)\mapsto 1-\beta c$.
\end{lemma}

\begin{proof}
	It is clear that the given formulas are inverse to each other.
	Suppose $u$ is a $\G_m$-orientation with associated unit $\beta$, and let $c=\beta^{-1}(1-u)\in \widehat E^1(\Pic)$. Then the formula for $\mathord\otimes^*(u)$ yields the desired formula for $\mathord\otimes^*(c)$. Moreover, since $u|_\mathrm{pt}=1$, we have $c|_{\P^1}=\beta^{-1}(1-u|_{\P^1})=\beta^{-1}(0,\beta)=(0,1)$, so that $c$ is an orientation.
	
	Conversely, let $(\beta,c)$ be a pair as in the statement and let $u=1-\beta c$. Then $u|_{\mathrm{pt}}=1-\beta\cdot 0=1$ and
	\[
	\mathord\otimes^*(u)=1-\beta\cdot \mathord\otimes^*(c)=1-\beta(x_1+x_2-\beta x_1x_2)=(1-\beta x_1)(1-\beta x_2)=u_1u_2,
	\]
	so that $u$ is a $\G_m$-preorientation. Moreover, $1-u|_{\P^1}=\beta c|_{\P^1}=(0,\beta)$ and $\beta$ is a unit.
\end{proof}

\begin{example}
	Let $\KGL\in\CAlg(\MotSp_S)$ be the motivic spectrum representing algebraic K-theory.
	The class
	\[
	u=[\scr O(-1)]\in \lim_n \K_0(\P^n) = \widehat{\KGL}{}^0(\Pic)
	\]
	is a $\G_m$-preorientation of $\KGL^0(\ph)\colon \MotSp_S^{\omega,\op}\to \Ab$ (as one sees using the Segre embeddings). The induced element $\beta=1-u|_{\P^1}\in \KGL^{-1}(\1)$ is the usual Bott element, given by the structure sheaf of the point $\infty\in\P^1$. Since $\beta$ is a unit, $\KGL^0(\ph)$ is $\G_m$-oriented.
\end{example}

\begin{proposition}[Universality of $\MGL$-cohomology]
	\label{prop:MGL-cohomology}
	Let $S$ be qcqs derived scheme. Then the ring cohomology theory
	\[
	\MGL^0(\ph)\colon \MotSp_S^{\omega,\op}\to \Ab
	\]
	is the initial object in the category of oriented cohomology theories on $\MotSp_S^\omega$.
\end{proposition}

\begin{proof}
The proof is a straightforward modification of the one of Theorem~\ref{thm:universality}.
The point is that, if $E^0$ is an oriented cohomology theory on $\MotSp_S^\omega$, then $\widehat{E}^{**}$ has the correct formula for $\Vect_n$:
\[
	\widehat{E}^{**}(X\times \Vect_n) \simeq \widehat{E}^{**}(X)[[c_1,\dotsc,c_n]].
\]
This follows from Theorem~\ref{thm:Gr-Vect} and the computation of the cohomology of Grassmannians using the projective bundle formula (see \cite[Lemma 4.5]{AnnalaIwasa}).
Then the multiplicativity of Thom classes in $\widehat{E}^{**}$ follows as in Lemma~\ref{lem:Whitney}, and we get the infinite Thom isomorphism
\[
	\widehat{E}^{**}(\Vect_\infty) \simeq \widehat{E}^{**}(\MGL)
\]
as in Proposition~\ref{prop:MGL}.
The cohomology class $1\in\widehat{E}^0(\Vect_\infty)\simeq\widehat{E}^0(\MGL)$ then gives the desired unique morphism $\MGL^0\to E^0$ of oriented cohomology theories as in Theorem~\ref{thm:universality}.
\end{proof}

\begin{corollary}[Universality of $\PMGL$-cohomology]\label{cor:PMGL-cohomology}
	Let $S$ be qcqs derived scheme. Then the ring cohomology theory
	\[
	\PMGL^0(\ph)\colon \MotSp_S^{\omega,\op}\to \Ab
	\]
	is the initial object in the category of periodic oriented cohomology theories on $\MotSp_S^\omega$.
\end{corollary}

\begin{proof}
	This follows immediately from Proposition~\ref{prop:MGL-cohomology}.
\end{proof}

\begin{lemma}\label{lem:MGL-comod}
	Let $\scr C$ be a symmetric monoidal cocomplete stable $\infty$-category whose tensor product preserves colimits in each variable and whose unit is compact.
	Let $E\in\CAlg(\h\scr C)$ be such that there is an isomorphism of $E$-modules
	\[
	E\otimes E\simeq \bigoplus_{\alpha}E\otimes L_\alpha
	\]
	with $L_\alpha\in \Pic(\scr C)$. Then the Amitsur complex of $E$ defines a $\Pic(\scr C)$-graded Hopf algebroid $(E_{\star}, E_{\star}E)$ such that the functors
	\begin{align*}
		E_{\star}(\ph)&\colon \scr C\to \Mod_{E_{\star}},\\
		E^{\star}(\ph)&\colon \scr C^{\omega,\op}\to\Mod_{E_{\star}}
	\end{align*}
	factor through the category of $(E_{\star}, E_{\star}E)$-comodules.
\end{lemma}

\begin{proof}
	The assumption on $E$ implies that, for any $X\in\scr C$ and $n\geq 1$, the canonical map
	\[
	(E^{\otimes n})_{\star}\otimes_{E_{\star}} E_{\star}X \to (E^{\otimes n})_{\star}X
	\]
	is an isomorphism, and that when $X$ is compact, the canonical map
	\[
	(E^{\otimes n})_{\star}\otimes_{E_{\star}} E^{\star}X \to (E^{\otimes n})^{\star}X
	\]
	is an isomorphism. Let $\eta_L\colon E\simeq E\otimes\1\to E\otimes E$ be the left unit.
	Taking $X=E$ itself and $n\leq 3$ yields the Hopf algebroid $(E_{\star},E_{\star}E)$ with comultiplication
	\[
	E_{\star}E \xrightarrow{\eta_L} (E\otimes E)_{\star}E 
	\xleftarrow\sim E_{\star}E\otimes_{E_{\star}} E_{\star}E.
	\]
	The coaction on $E_{\star}X$ is then given by the composite
	\[
	E_{\star}X \xrightarrow{\eta_L} (E\otimes E)_{\star}X
	\xleftarrow\sim E_{\star}E\otimes_{E_{\star}} E_{\star}X,
	\]
	and the coaction on $E^{\star}X$ for $X$ compact is given by the composite
	\[
	E^{\star}X \xrightarrow{\eta_L} (E\otimes E)^{\star}X
	\xleftarrow\sim E_{\star}E\otimes_{E_{\star}} E^{\star}X.\qedhere
	\]
\end{proof}

By Corollary~\ref{cor:homology_MGL}, we have an isomorphism of $\MGL$-modules
\[
\MGL\otimes \MGL\simeq \MGL[b_1,b_2,\dotsc]=\bigoplus_{m}\Sigma_{\P^1}^{\deg (m)}\MGL,
\]
where $m$ ranges over the monomials in the variables $b_i$ and $\deg(b_i)=i$.
Lemma~\ref{lem:MGL-comod} therefore applies to $\MGL\in \CAlg(\h\MotSp_S)$ and yields a $\Z$-graded Hopf algebroid $(\MGL_*,\MGL_*\MGL)$.\footnote{This Hopf algebroid is a priori only $\tau_{\leq 1}\mathbb S$-graded, but it turns out to be $\Z$-graded as the swap map on $\P^1\otimes\P^1$ induces the identity on $\MGL_2(\ph)$, by the naturality of the Thom isomorphism.}
Note that if $c$ and $c'$ are two orientations of $E\in\CAlg(\h\MotSp_S)$, then there is a unique power series $f(t)\in t+t^2E_*[[t]]$ such that $f(c)=c'$ in $E^1(\Pic)$, which defines a strict isomorphism between the associated formal group laws over $E_*$.
The graded formal group law $F$ over $\MGL_*$ and the strict isomorphism between the two formal group laws $\eta_L^*(F)$ and $\eta_R^*(F)$ over $\MGL_*\MGL$ then induce a morphism of graded Hopf algebroids
\[
(\L,\LB)\to (\MGL_*,\MGL_*\MGL),
\]
where $(\L,\LB)$ is the Hopf algebroid classifying the strict groupoid of formal group laws and strict isomorphisms.
Recall that $\LB$ is a polynomial ring $\L[b_0,b_1,b_2,\dotsc]/(b_0-1)$, over which the power series $\sum_{i\geq 0}b_i x^{i+1}$ is the universal strict isomorphism \cite[Proposition A2.1.15]{Ravenel:2003}.
Proposition~\ref{prop:homology_MGL}\ref*{item:hom-MGL} implies that the above morphism is a cocartesian natural transformation of cosimplicial commutative rings, so that a structure of $(\MGL_*,\MGL_*\MGL)$-comodule on an $\MGL_*$-module is equivalent to a structure of $(\L,\LB)$-comodule on the underlying $\L$-module.

\begin{proposition}\label{prop:universal-mult-MGL}
	Let $S$ be qcqs derived scheme. Then the ring cohomology theory
	\[
	(\MGL^*(\ph)\otimes_\L\Z[\beta^{\pm 1}])_{0}\colon \MotSp_S^{\omega,\op}\to \Ab
	\]
	is the initial object in the category of $\G_m$-oriented cohomology theories on $\MotSp_S^\omega$.
\end{proposition}

\begin{proof}
	By Lemma~\ref{lem:MGL-comod}, the functor $\MGL^*(-)$ on compact spectra is valued in $(\L,\LB)$-comodules. Since $\Z[\beta^{\pm 1}]$ is a flat $(\L,\LB)$-comodule by Landweber's criterion \cite[Lecture 15, Example 12]{Lurie:2010}, the given functor is indeed a homological functor. It then follows from Lemma~\ref{lem:G_m-orientation} that it has the stated universal property.
\end{proof}

\begin{lemma}\label{lem:universal-pre-G_m}
	Let $S$ be qcqs derived scheme. Then the ring cohomology theory
	\[
	(\Sigma^\infty_{\P^1}\Pic_+)^0(\ph)\colon \MotSp_S^{\omega,\op}\to \Ab
	\]
	is the initial object in the category of $\G_m$-preoriented cohomology theories on $\MotSp_S^\omega$.
\end{lemma}

\begin{proof}
	This is clear by Remark~\ref{rmk:E-hat}.
\end{proof}

\begin{proposition}[Universality of $\KGL$-cohomology]
	\label{prop:universal-mult-K}
	Let $S$ be qcqs derived scheme. Then the ring cohomology theory
	\[
	\KGL^0(\ph)\colon \MotSp_S^{\omega,\op}\to \Ab
	\]
	is the initial object in the category of $\G_m$-oriented cohomology theories on $\MotSp_S^\omega$.
\end{proposition}

\begin{proof}
	By \cite[Theorem 5.3.3]{AnnalaIwasa2}, there is an isomorphism of motivic $\E_\infty$-ring spectra \[\KGL\simeq \Sigma^\infty_{\P^1}\Pic_+[\beta^{-1}],\]
	where $\beta=1-[\scr O(-1)]$.
	This is the Bott element associated with the $\G_m$-preorientation on $(\Sigma^\infty_{\P^1}\Pic_+)^0(\ph)$ given by the dual of the universal invertible sheaf, which is a universal $\G_m$-preorientation by Lemma~\ref{lem:universal-pre-G_m}.
	Hence, the cohomology theory defined by $\Sigma^\infty_{\P^1}\Pic_+[\beta^{-1}]$ has the desired universal property.
\end{proof}

\begin{theorem}[Algebraic Conner--Floyd isomorphism]\label{thm:Conner-Floyd}
	Let $X$ be a qcqs derived scheme. Then there is an isomorphism of bigraded rings
	\[
	\MGL^{**}(X)\otimes_{\L}\Z[\beta^{\pm 1}] \simeq \KGL^{**}(X).
	\]
\end{theorem}

\begin{proof}
	Combine Propositions \ref{prop:universal-mult-MGL} and~\ref{prop:universal-mult-K}.
\end{proof}

\begin{theorem}[Rational $\PMGL$-cohomology]
	\label{thm:QPMGL}
Let $X$ be a qcqs derived scheme. Then there is an isomorphism of bigraded rings 
\[
\Q \otimes \PMGL^{**}(X) \simeq \L_\Q  \otimes \KGL^{**}(X),
\]
where $\L_\Q$ is shorthand for $\Q \otimes \L$.
\end{theorem}

\begin{proof}
Both cohomology theories $\Q \otimes \PMGL^{*}(\ph)$ and $\L_\Q  \otimes \KGL^{*}(\ph)$ take values in graded $\L_\Q[u^{\pm 1}]$-algebras, where the degree $-1$ element $u$ acts as the canonical unit on $\Q \otimes \PMGL^{*}(\ph)$ and as $\beta$ on $\L_\Q  \otimes \KGL^{*}(\ph)$. Let $c_\Omega$ and $c_\K$ be the orientations of $\Q \otimes \PMGL^{0}(\ph)$ and $\KGL^{0}(\ph)$, respectively. By abuse of notation, we will also denote by $c_\K$ the orientation $1 \otimes c_\K$ of $\L_\Q  \otimes \KGL^{0}(\ph)$.

Since all formal group laws with coefficients in a ring containing the rationals are equivalent \cite[Theorem~1.6.2]{hazewinkel:1978}, there exists a power series $f(t) \in t + t^2 \L_\Q[u][[t]]$ such that the orientation $f(c_\Omega)$ of $\Q \otimes \PMGL^{0}(\ph)$ satisfies the formal group law $x + y - \beta x y$, and the orientation $f^{-1}(c_\K)$ of $\L_\Q \otimes \KGL^{0}(\ph)$ satisfies the universal formal group law of $\L_\Q$. Here, $f^{-1}(t)$ denotes the compositional inverse of $f(t)$. By Corollary~\ref{cor:PMGL-cohomology} and Proposition~\ref{prop:universal-mult-K} and extension of scalars, there exist unique morphisms of oriented cohomology theories
\[
\Psi\colon \left(\Q \otimes \PMGL^{0}, c_\Omega\right) \to  \left(\L_\Q \otimes \KGL^{0}, f^{-1}(c_\K) \right)
\]
and 
\[
\Phi\colon \left(\L_\Q \otimes \KGL^{0}, c_\K \right) \to \left(\Q \otimes \PMGL^{0}, f(c_\Omega)\right),
\]
the latter of which is required to be $\L_\Q$-linear. By the hypothesis on $f^{-1}(c_\K)$, $\Psi$ is also $\L_\Q$-linear. As $\Phi \circ \Psi$ is orientation-preserving, it is the identity by the universal property of $\PMGL^0(\ph)$. As $\Psi \circ \Phi$ is orientation-preserving and $\L_\Q$-linear, it is the identity by the universal property of $\KGL^0(\ph)$. 
\end{proof}

Recall that the \emph{universal precobordism ring} $\PCob^*(X)$ of a derived scheme $X$ is defined as the group completion of the monoid of equivalence classes $[V \to X]$ of projective quasi-smooth derived schemes over $X$, modulo the relations
\[
[W_0 \to X] = [A \to X] + [B \to X] - [\P_{A \cap B}(\scr O(A) \oplus \scr O)]
\]
for every quasi-smooth projective $W \to \P^1_X$ with fibers $W_0$ and $A+B$ over $0$ and $\infty$ respectively \cite{AY, annala-chern}. Here, $A+B$ denotes the sum of virtual effective Cartier divisors.

\begin{corollary}\label{cor:QMGLOmega}
Let $X$ be a noetherian derived scheme of finite Krull dimension, and assume that $X$ admits an ample line bundle\footnote{By employing a slightly more complicated construction of $\PCob^*$, it is possible to weaken the assumptions on $X$ to merely admitting an ample family of line bundles, see \cite{annala-base-ind-cob}.}. Then there is a natural isomorphism of rings
\[
\Q \otimes \bigoplus_{n\in\Z}\MGL^n(X) \simeq \Q \otimes \bigoplus_{n\in\Z}\PCob^n (X).
\]
\end{corollary} 
\begin{proof}
Combine Theorem~\ref{thm:QPMGL} with \cite[Theorem~236]{AnnalaThesis}.
\end{proof}

\begin{remark}
Periodization loses track of the grading, and therefore we do not immediately obtain an isomorphism of graded rings. However, by constructing enough transfers for $\MGL^*$, it would be possible to obtain a comparison map $\PCob^* \to \MGL^*$ of graded rings with integer coefficients \cite[Theorem~192]{AnnalaThesis}. It is an interesting question under which conditions this map, or rather its refinement $\Omega^* \to \MGL^*$, where $\Omega^*$ is the \emph{derived algebraic cobordism} \cite{annala-pre-and-cob}, is an isomorphism. The only known instance seems to be Levine's result \cite{Levine:2009}, which states that the natural map $\Omega^*(X) \to (\L_{\A^1} \MGL)^{*}(X)$ is an isomorphism for all schemes $X$ that are smooth and quasi-projective over a field of characteristic 0.
\end{remark}

\begin{remark}[Conner--Floyd isomorphism for Selmer K-theory]
Let $\MotSp^\et_{S}$ and $\MotSp_S^{\et,\hyp}$ be the full subcategories of $\MotSp_S$ spanned by the étale sheaves and the étale hypersheaves, respectively. The image of $\KGL$ in $\MotSp^\et_S$ then represents the Zariski sheafification of Selmer K-theory $\K^\Sel$, see \cite[Section 5.4]{AnnalaIwasa2}. If $S$ is qcqs of finite Krull dimension and of finite punctual étale cohomological dimension, then $\MotSp_S^{\et,\hyp}$ is compactly generated and the localization functor $\MotSp_S\to\MotSp^{\et,\hyp}_S$ preserves compact objects (combine \cite[Corollary 3.29]{ClausenMathew} and \cite[Lemma 2.16]{BachmannRigidity}). In this case, $\K^\Sel$ is also an étale hypersheaf on $\Sm_S^\fp$ \cite[Corollary 7.15]{ClausenMathew}.
One can easily see that the arguments in this section go through if we replace $\MotSp_S$ by any commutative $\MotSp_S$-algebra in $\Pr^{\L}_{\omega}$ \cite[Notation 5.5.7.7]{HTT}. Under this finiteness assumption on $S$, we therefore obtain an isomorphism of bigraded rings
\[
	\MGL^{\et,\hyp **}(S)\otimes_{\L}\Z[\beta^{\pm 1}] \simeq \K^{\Sel **}(S).
\]
\end{remark}

\section{Snaith theorem for periodic algebraic cobordism}

We prove the Snaith theorem for $\PMGL$, which is a non-$\A^1$-localized refinement of a theorem of Gepner--Snaith \cite[Corollary 3.10]{GepnerSnaith}.
Our proof is however quite different from theirs\footnote{In fact, the argument in \emph{loc.\ cit.}\ seems to contain a crucial mistake: the proof of \cite[Theorem 3.9]{GepnerSnaith} uses a universal property of localization of ring spectra in the homotopy category, which is not valid.} and uses instead the same strategy as the proof of the Conner--Floyd isomorphism in Section~\ref{sec:CF}.

For the following definition, we recall that the map $\Vect_\infty\to \K_{\rk=0}$ becomes an isomorphism in the $\infty$-category of Zariski sheaves of spectra (see the proof of Proposition~\ref{prop:MGL-colim}). In particular, $\SigmaP^\infty\Vect_{\infty,+}$ has a canonical structure of $\E_\infty$-algebra in $\MotSp_S$, whose multiplication we denote by $\oplus$.

\begin{definition}\label{def:GL-ori}
	Let $S$ be a qcqs derived scheme.
A \textit{$t$-preoriented cohomology theory} on $\MotSp_S^\omega$ is a ring cohomology theory $E^0$ with an element $u\in\widehat{E}^0(\Vect_\infty)$ such that $u\vert_\rm{pt}=1$ and $\oplus^*(u)=u_1 u_2$, where $u_i=\pi_i^*(u)\in\widehat{E}^0(\Vect_\infty\times\Vect_\infty)$.
We then define the \textit{Bott element} $\beta\in E^{-1}(\1)=\tilde{E}^0(\P^1)$ to be the element $1-u\vert_{\P^1}$, and we say that $E^0$ is \emph{$t$-oriented} if $\beta$ is a unit.
\end{definition}

Let $E^0$ be an oriented cohomology theory on $\MotSp_S^\omega$.
If $\scr E_n$ is the universal rank $n$ locally free sheaf on $\Vect_n$, then the restriction of the Chern class $c_i(\scr E_n)\in\widehat{E}^i(\Vect_n)$ to $\Vect_{n-1}$ is the Chern class $c_i(\scr E_{n-1})$. Thus, the sequence $(c_i(\scr E_n))_{n\geq 0}$ defines a canonical element
\[
c_i\in \lim_n \widehat{E}^i(\Vect_n)=\widehat{E}^i(\Vect_\infty),
\]
such that for any $X\in\scr P(\Sm_S)$ we have
\[
\widehat{E}^{**}(X\times \Vect_\infty)\simeq\widehat{E}^{**}(X)[[c_1,c_2,\dotsc]].
\]

The following lemma is the analogue of Lemma~\ref{lem:G_m-orientation} for $t$-orientations:

\begin{lemma}[Orientations vs.\ $t$-orientations]
	\label{lem:GL-ori}
Let $E^0$ be a ring cohomology theory on $\MotSp_S^\omega$.
Then the assignment
\[
	u\mapsto (\beta,c), \quad \beta=1-u\vert_{\P^1}, \quad c=\beta^{-1}(1-u\vert_{\Pic})
\]
gives a bijection between $t$-orientations $u$ of $E^0$ and pairs $(\beta,c)$ consisting of a unit $\beta\in E^{-1}(\1)$ and an orientation $c\in\widehat E^1(\Pic)$.
The inverse is given by the formula
\[
u = \sum_{i\ge 0}(-\beta)^ic_i \in \widehat{E}^0(\Vect_\infty).
\]
\end{lemma}

\begin{proof}
Suppose $u$ is a $t$-orientation with associated unit $\beta$, and let $c=\beta^{-1}(1-u\vert_{\Pic})\in\hat E^1(\Pic)$.
Since $u|_\mathrm{pt}=1$, we have $c|_{\P^1}=\beta^{-1}(1-u|_{\P^1})=\beta^{-1}(0,\beta)=(0,1)$, so that $c$ is an orientation.
Conversely, suppose that $E^0$ is perodic and oriented, with Bott element $\beta$ and orientation $c$, and let $u=\sum_{i}(-\beta)^ic_i$.
Then it is clear that $u\vert_\rm{pt}=1$ and that $\beta=1-u\vert_{\P^1}$. Furthermore, it follows from the Whitney sum formula that $\oplus^*(u)=u_1u_2$, so that $u$ is a $t$-orientation.
 
It remains to show that the two assignments are inverse to each other.
It is clear that the composite $(\beta,c)\mapsto u\mapsto (\beta,c)$ is the identity.
Conversely, given a $t$-orientation $u$, we have to prove the equality $u = \sum_{i\ge 0}(-\beta)^ic_i$.
Since $\widehat{E}^0(\Vect_\infty)=\lim_n\widehat{E}^0(\Vect_n)$, it suffices to show that these two elements coincide in $\widehat{E}^0(\Vect_n)$ for every $n$. This is clear for $n=1$.
Note that the map
\[
	\widehat{E}^*(\Vect_n) \to \widehat{E}^*(\Pic^n)
\]
induced by the direct sum $\oplus\colon\Pic^n\to \Vect_n$ is injective.
We now conclude by observing that the two elements coincide in the right-hand side, by the case $n=1$, the formula $\oplus^*(u)=u_1u_2$, and the Whitney sum formula for Chern classes.
\end{proof}

\begin{theorem}[Snaith theorem for $\PMGL$]\label{thm:Snaith}
For any derived scheme $S$, there is a canonical isomorphism
\[
	\PMGL \simeq \SigmaP^\infty\Vect_{\infty,+}[\beta^{-1}]
\]
in $\CAlg(\h\MotSp_S)$, where $\beta=1-[\scr O(-1)]$.
\end{theorem}

\begin{proof}
	Let $u\in (\SigmaP^\infty\Vect_{\infty,+})^0(\Vect_\infty)$ be the element induced by minus the universal K-theory element of rank $0$, so that $\beta=1-u|_{\P^1}$.
	Then $\beta^{-1}(1-u|_{\Pic})$ is an orientation of the periodic ring spectrum $\SigmaP^\infty\Vect_{\infty,+}[\beta^{-1}]$.
By the universal property of $\PMGL$ (Corollary~\ref{cor:universality}), we obtain a canonical map 
\[\PMGL\to\SigmaP^\infty\Vect_{\infty,+}[\beta^{-1}]\]
in $\CAlg(\h\MotSp_S)$.
To prove that it is an isomorphism, we may assume $S$ qcqs. In this case, $u$ defines a $t$-preorientation of $(\SigmaP^\infty\Vect_{\infty,+})^0(\ph)$ with associated Bott element $\beta$, which by Remark~\ref{rmk:E-hat} is the initial $t$-preorientation. By Corollary~\ref{cor:PMGL-cohomology} and Lemma~\ref{lem:GL-ori}, both sides then have the same universal property as cohomology theories on $\MotSp_S^\omega$, and it follows that the map is an isomorphism.
\end{proof}

\begin{remark}\label{rmk:det}
	Under the Snaith isomorphisms
	\begin{align*}
		\PMGL&\simeq \SigmaP^\infty\Vect_{\infty,+}[\beta^{-1}],\\
		\KGL&\simeq \SigmaP^\infty\Pic_+[\beta^{-1}]
	\end{align*}
	of Theorem~\ref{thm:Snaith} and \cite[Theorem 5.3.3]{AnnalaIwasa2}, the orientation map $\PMGL\to\KGL$ in $\CAlg(\h\MotSp_S)$ provided by Corollary~\ref{cor:universality} is induced by the determinant $\det\colon \K_{\rk=0}\to\Pic$. This follows from the fact that $\SigmaP^\infty\det_+$ sends $u|_{\Pic}\in (\SigmaP^\infty\Vect_{\infty,+})^0(\Pic)$ to the class in $(\SigmaP^\infty\Pic_+)^0(\Pic)$ represented by the dual of the universal invertible sheaf.
\end{remark}

\begin{remark}
Both $\PMGL$ and $\SigmaP^\infty\Vect_{\infty,+}[\beta^{-1}]$ have canonical $\E_\infty$-algebra structures, but they are not isomorphic as $\E_\infty$-algebras in general, since they are known not to be isomorphic as $\E_5$-algebras after Betti realization \cite[Theorem 1.4]{hahn:2020}.
One might expect that they are at least isomorphic as $\E_1$-algebras, but this is not known even after $\A^1$-localization.

Note that the determinant induces an $\E_\infty$-map
\[
\SigmaP^\infty\Vect_{\infty,+}[\beta^{-1}]\to\KGL
\]
(see Remark~\ref{rmk:det}). At this point we do not know if there is also an $\E_\infty$-map
\[
\PMGL\to\KGL,
\]
although this is known in $\A^1$-homotopy theory using the formalism of framed correspondences \cite[Proposition 6.2]{motive-hilb}.
\end{remark}

\bibliographystyle{alphamod}
\bibliography{references}
    
\end{document}